\documentclass[11pt,a4paper]{article}
\usepackage[T1]{fontenc}
\usepackage[utf8]{inputenc}
\usepackage{lmodern}
\usepackage[margin=28mm]{geometry}
\usepackage{amsmath,amssymb,amsthm,mathtools}
\usepackage{booktabs,array,microtype}
\usepackage{xcolor}
\usepackage[colorlinks=true,linkcolor=blue!55!black,citecolor=blue!55!black,urlcolor=blue!55!black]{hyperref}
\hypersetup{pdftitle={Large cyclic automorphism groups and cyclic subgroups of index two in characteristic two},pdfauthor={Marco Timpanella}}
\newtheorem{theorem}{Theorem}[section]
\newtheorem{proposition}[theorem]{Proposition}
\newtheorem{lemma}[theorem]{Lemma}
\newtheorem{corollary}[theorem]{Corollary}
\theoremstyle{definition}
\newtheorem{example}[theorem]{Example}
\newtheorem{remark}[theorem]{Remark}
\DeclareMathOperator{\Aut}{Aut}
\DeclareMathOperator{\lcm}{lcm}
\DeclareMathOperator{\ord}{ord}
\newcommand{\PP}{\mathbb P}
\newcommand{\FF}{\mathbb F}
\newcommand{\XX}{\mathcal X}
\newcommand{\YY}{\mathcal Y}
\newcommand{\ZZ}{\mathcal Z}
\newcommand{\NN}{\mathrm N}
\numberwithin{equation}{section}
\title{Large cyclic automorphism groups and cyclic subgroups\newline of index two in characteristic two}
\author{Marco Timpanella}
\date{}

\begin{document}
\maketitle
\begin{abstract}
Let $\XX$ be a projective, geometrically irreducible, nonsingular algebraic curve of genus $g\ge2$ over an algebraically closed field of characteristic two. We classify cyclic subgroups of $\Aut(\XX)$ of order $N\ge2g+1$. Besides Kummer extensions of odd degree ramified over three points, precisely two cases occur: $y^2+y=x^m$, with $N=2m=4g+2$, and $y^2+y=c/(x^m+1)$, with $m$ odd, $c\ne0$, and $N=2m=2g+2$. In particular, $4\nmid N$. We then classify groups $H$ with a cyclic subgroup of index two and $|H|>4g+4$. They are precisely the groups $C_k\times D_{2M}$ on the curves $y^k=x^M+x^{-M}$, where $k,M\ge3$ are odd and coprime, and $2g=M(k-1)$. Their possible orders are $4g+2M$, where $M$ runs over certain odd divisors of $g$, and $|H|\le6g$. No curve in this family is ordinary. Hermitian curves occur in the family, and we give two sufficient conditions for maximality over finite fields. For dihedral groups the sharp bound is $4g+4$ in even genus and $4g$ in odd genus. The equality cases are given by explicit Artin--Schreier equations, and in each case the dihedral group is the full automorphism group.
\end{abstract}

\section{Introduction and main results}

Throughout the paper, $K$ is an algebraically closed field of characteristic two and $\XX$ is a projective, geometrically irreducible, nonsingular algebraic curve defined over $K$, of genus $g=g(\XX)\ge2$. We denote by $K(\XX)$ its function field and by $\Aut(\XX)$ the group of $K$-automorphisms of $K(\XX)$.

A cyclic subgroup of $\Aut(\XX)$ whose order is large with respect to $g$ imposes strong restrictions on the curve. The reason is that the Riemann--Hurwitz formula leaves little room for the genus of the quotient and for ramification. At the threshold $N\ge2g+1$, the quotient by a cyclic group of order $N$ is rational and only a small number of ramification configurations can occur. The problem is therefore sufficiently rigid to admit a classification, but it still reflects the difference between tame and wild ramification.

In characteristic zero, possible large orders were studied by Nakagawa \cite{Nakagawa}, and Irokawa and Sasaki classified cyclic groups of order at least $2g+1$ \cite{IS}. A related question is whether a large cyclic group is contained in a larger automorphism group. This extension problem was studied by Singerman and Watson \cite{SW} and by Bujalance and Conder \cite{BC}. For dihedral groups the sharp bound is $4g+4$ when $g$ is even and $4g$ when $g$ is odd; see \cite{BCGG}. We refer to \cite{Broughton} for a survey of automorphism groups of algebraic curves in characteristic zero.

In positive characteristic the tame part of this picture remains essentially unchanged, whereas the wild part produces new curves. Some earlier results on cyclic groups in positive characteristic can be found in \cite{Homma,Sanjeewa}; despite its title, \cite{Sanjeewa} assumes that the characteristic is different from two. In \cite{DGTcyclic}, cyclic groups of order at least $2g+1$ were classified in odd characteristic. In \cite{DGTindex}, groups with a cyclic subgroup of index two were classified above the bound $4g+4$, and the dihedral bound was proved in odd characteristic. The present paper treats the remaining characteristic-two case. Together with \cite{DGTcyclic,DGTindex} and the classical results, it completes the classification at these thresholds over algebraically closed fields of arbitrary characteristic. In characteristic two, Giulietti and Korchm\'aros constructed ordinary curves of genus $g=2^h+1$ with a dihedral group of order $4(g-1)$ acting without fixed points \cite{GK2}; these examples lie just below the odd-genus bound proved here.

The case of characteristic two is not obtained by simply putting $p=2$ in the odd-characteristic classification. If the cyclic order is even, its involution is wild and the corresponding quadratic extension is Artin--Schreier. Moreover, an involution of a rational curve has one fixed point in characteristic two, while it has two fixed points in odd characteristic. Both differences affect the possible ramification configurations. There is also a difference in the genus of the curves $y^k=x^M+x^{-M}$: in odd characteristic the same equation has twice the genus and the group $C_k\times D_{2M}$ remains below the bound $4g+4$. The reason is the identity $x^{2M}+1=(x^M+1)^2$ in characteristic two; see Remark~\ref{rem:char2phenomenon}.

Grothendieck's tame lifting theorem gives a useful comparison with the classical case. When the cyclic group has odd order, its action can be lifted to characteristic zero, and the corresponding part of our classification could also be deduced from the theorem of Irokawa and Sasaki; see \cite[Expos\'e~XIII]{SGA1}. This observation does not substantially simplify the characteristic-two problem. The even cyclic actions considered here are wild, the involution occurring in the index-two problem gives wild ramification, and every dihedral group has even order. Thus the parts of the classification that are not already covered by the classical theory are precisely those to which tame lifting does not apply. They require a direct study of wild ramification.

Our first result completes the classification of large cyclic groups in characteristic two. The tame case has the same form as in characteristic zero and odd characteristic: the quotient is rational, the cover ramifies over three points, and the curve has a Kummer equation
\[
 y^N=x^r(x+1)^s.
\]
The even case is more rigid. We prove that $4\nmid N$ and that only the two Artin--Schreier curves in cases~\textup{(II)} and~\textup{(III)} of Theorem~\ref{thm:cyclic} occur. Thus the obstruction coming from characteristic two does not produce a long list of exceptional cases; it produces exactly two families.

We then turn to groups $H$ having a cyclic subgroup $G$ of index two. When $|H|>4g+4$, the group $G$ is necessarily odd. Its cover has three ramification points. The involution induced by $H/G$ on the rational quotient fixes one of them and exchanges the other two. This simple observation, combined with the exponents of the Kummer equation, determines both the curve and the group. We obtain
\[
 \XX_{k,M}:\quad y^k=x^M+x^{-M},
 \qquad H\cong C_k\times D_{2M},
\]
where $k,M\ge3$ are odd and coprime. This gives the exact spectrum of possible orders above $4g+4$ and the uniform bound $|H|\le6g$. It also leads to sharp bounds for dihedral groups: $4g+4$ when $g$ is even and $4g$ when $g$ is odd. Although these are also the sharp bounds in characteristic zero, the proof here is different because every reflection is wild. We also determine the extremal curves explicitly and prove that their dihedral groups are their full automorphism groups.

The explicit equation of $\XX_{k,M}$ allows us to go beyond the classification. We determine the ramification of the displayed group, prove that these curves are not ordinary, identify the Hermitian curves occurring in the family, and give two sufficient conditions for maximality over finite fields. In this way the classification also produces a concrete family whose arithmetic and geometric properties can be studied directly.

We write $C_n$ for a cyclic group of order $n$ and use the convention
\[
 D_{2n}=\langle r,t\mid r^n=t^2=1,\ trt=r^{-1}\rangle,
\]
for the dihedral group of order $2n$.

We write $\gamma(\XX)$ for the $2$-rank, or equivalently the Hasse--Witt invariant, of $\XX$. As usual, $0\le\gamma(\XX)\le g$, and $\XX$ is ordinary if and only if $\gamma(\XX)=g$. For a finite subgroup $G\le\Aut(\XX)$ we denote by $\XX/G$ the quotient curve; a $G$-orbit with fewer than $|G|$ points is called a short orbit.

\begin{theorem}\label{thm:cyclic}
Let $G\le\Aut(\XX)$ be cyclic of order $N\ge2g+1$. Then $\XX/G$ is rational and, up to a $K$-isomorphism of $\XX$ carrying $G$ onto the displayed cyclic group, one of the following three cases occurs.
\begin{enumerate}
\item[\textup{(I)}] $N$ is odd and
\begin{equation}\label{eq:tame}
 \XX:\quad y^N=x^r(x+1)^s,
 \qquad r,s\ge1,\quad r+s\le N-1,\quad \gcd(N,r,s)=1.
\end{equation}
Here $G$ is generated by $(x,y)\mapsto(x,\zeta_N y)$ and
\begin{equation}\label{eq:tamegenus}
 2g=N+2-\gcd(N,r)-\gcd(N,s)-\gcd(N,r+s).
\end{equation}
\item[\textup{(II)}] $N=2m$, with $m\ge5$ odd, and
\begin{equation}\label{eq:onepole}
 \XX:\quad y^2+y=x^m,\qquad g=(m-1)/2.
\end{equation}
Here $\gamma(\XX)=0$ and $G$ is generated by $(x,y)\mapsto(\zeta_m x,y+1)$.
\item[\textup{(III)}] $N=2m$, with $m\ge3$ odd, and
\begin{equation}\label{eq:ordinary}
 \XX:\quad y^2+y=\frac{c}{x^m+1},\qquad c\in K^*,\qquad g=m-1.
\end{equation}
The curve is ordinary, and $G$ has the same generator as in case~\textup{(II)}.
\end{enumerate}
Conversely, every curve and subgroup in the list satisfies $N\ge2g+1$. In particular, $4\nmid N$.
\end{theorem}

The exponents $(r,s,-r-s)$ in case~\textup{(I)} are considered modulo $N$. A permutation of the three points of $\XX/G$ over which the cover ramifies, or multiplication of all three exponents by an integer coprime to $N$, may give another equation for the same curve and cyclic group; no uniqueness of $(r,s)$ is asserted.

\begin{theorem}\label{thm:index}
Let $H\le\Aut(\XX)$ contain a cyclic subgroup $G$ of index two, and assume $|H|>4g+4$. Then $|G|$ is odd. There exist odd coprime integers $k,M\ge3$ such that
\begin{equation}\label{eq:family}
 \XX\cong\XX_{k,M}:\quad y^k=x^M+x^{-M},
 \qquad 2g=M(k-1),
\end{equation}
and, under this isomorphism,
\begin{equation}\label{eq:Haction}
 H=\langle a,b,t\rangle\cong C_k\times D_{2M},
\end{equation}
where
\[
 a:(x,y)\mapsto(x,\zeta_k y),\qquad
 b:(x,y)\mapsto(\zeta_M x,y),\qquad
 t:(x,y)\mapsto(x^{-1},y),
\]
and $G=\langle a,b\rangle\cong C_{kM}$. Conversely, every such pair $(k,M)$ gives a group satisfying the hypotheses.
\end{theorem}
For comparison, Reyes-Carocca and Speziali determined the full automorphism group of certain complex curves admitting a group of order $6g$; in those examples it is $C_3\times D_{2g}$ \cite{RS}; see also \cite[Remark~5.7]{DGTindex}. Theorem~\ref{thm:index} gives a direct-product structure for every group above our threshold in characteristic two, although the full group $\Aut(\XX)$ may be larger. In contrast, Theorem~\ref{thm:dihedral} shows that the extremal dihedral curves have no further automorphisms.

\begin{corollary}\label{cor:spectrumintro}
For fixed $g\ge2$, the set of possible orders greater than $4g+4$ is
\begin{equation}\label{eq:spectrum}
 \mathcal S_2(g)=
 \left\{4g+2M:\ M\mid g,\quad M\ge3\text{ odd},\quad
 \gcd\left(\frac{2g}{M}+1,M\right)=1\right\}.
\end{equation}
Every curve in Theorem~\ref{thm:index} is nonhyperelliptic and $|H|\le6g$. Equality holds if and only if $g\ge5$ is odd, $3\nmid g$, and $\XX$ is isomorphic to $\XX_{3,g}$, with $H$ identified under this isomorphism with the group generated by the maps in \eqref{eq:Haction}.
\end{corollary}

\begin{theorem}\label{thm:dihedral}
If $D\le\Aut(\XX)$ is dihedral, then
\[
 |D|\le\begin{cases}4g+4,&g\text{ even},\\4g,&g\text{ odd}.\end{cases}
\]
The equality cases are precisely the curves
\[
 \begin{array}{ll}
 y^2+y=c/(x^{g+1}+1),\quad c\in K^*,&g\text{ even},\\[2pt]
 y^2+y=c(x^g+x^{-g}),\quad c\in K^*,&g\text{ odd}.
 \end{array}
\]
For every curve in the first family, $\Aut(\XX)$ is the dihedral group of order $4g+4$ displayed in the proof. For every curve in the second family, $\Aut(\XX)$ is exactly the dihedral group of order $4g$. The first family is ordinary, while the second has $2$-rank one. Thus the respective bounds are attained in every even genus $g\ge2$ and every odd genus $g\ge3$.
\end{theorem}

The proofs are based on a direct analysis of short orbits and their stabilizers by means of the Riemann--Hurwitz formula. Once the possible ramification configurations have been determined, the corresponding extensions are written explicitly by Kummer or Artin--Schreier equations. The characteristic-two case requires a separate estimate for higher ramification groups in order to exclude cyclic orders divisible by four. The equation of the family in Theorem~\ref{thm:index} is then obtained by studying the exponents of a Kummer extension of a rational function field.

All the facts used repeatedly in the proofs are collected in the preliminaries. In particular, we include the precise forms of the Kummer and Artin--Schreier arguments needed later, so that the numerical computations in the classification can be followed directly.

\section{Preliminaries}

Our notation and terminology are standard; see \cite[Chapter~11]{HKT}, \cite{DGTcyclic,DGTindex}. As in the introduction, $K(\XX)$ denotes the function field of $\XX$.
Let $G\le\Aut(\XX)$ be finite. The quotient curve $\XX/G$ is the nonsingular model of the fixed field $K(\XX)^G$, and the extension $K(\XX)/K(\XX)^G$ is Galois of degree $|G|$. We denote by
\[
 \Phi:\XX\longrightarrow\XX/G
\]
the corresponding Galois cover.

For $P\in\XX$, let $G_P$ be the stabilizer of $P$ in $G$. The point $P$ is a ramification point if $G_P\ne\{1\}$, and its ramification index is
\[
 e_P=|G_P|.
\]
The orbit of $P$ has length $|G|/|G_P|$. Hence the ramification points are precisely the points belonging to short orbits. If $Q=\Phi(P)$, we say that the cover ramifies over $Q$. Thus a short orbit is exactly a fiber over which the cover ramifies.

As in \cite{DGTcyclic,DGTindex}, we write $G_P^{(i)}$ for the $i$th ramification group at $P$. If $u$ is a local parameter at $P$, then
\[
 G_P^{(i)}=\{\sigma\in G_P:\ord_P(\sigma(u)-u)\ge i+1\}.
\]
In particular, $G_P^{(0)}=G_P$. The different exponent and the Riemann--Hurwitz formula are
\begin{equation}\label{eq:RH}
 d_P=\sum_{i\ge0}(|G_P^{(i)}|-1),\qquad
 2g-2=|G|(2g(\XX/G)-2)+\sum_{P\in\XX}d_P.
\end{equation}
For each point $Q\in\XX/G$ over which the cover ramifies, choose $P\in\Phi^{-1}(Q)$, put $e_Q=|G_P|$, and set $\delta_Q=d_P/e_Q$. The fiber of $Q$ contains $|G|/e_Q$ points, and the different exponent is the same at all of them. Dividing \eqref{eq:RH} by $|G|$ therefore gives
\begin{equation}\label{eq:normalized}
 \frac{2g-2}{|G|}=2g(\XX/G)-2+\sum_Q\delta_Q.
\end{equation}
If $P$ is tamely ramified, then $G_P^{(i)}=\{1\}$ for $i\ge1$, and hence
\[
 d_P=e_Q-1,\qquad \delta_Q=1-\frac1{e_Q}.
\]
If $P$ is wildly ramified, then $|G_P^{(1)}|\ge2$. The terms corresponding to $i=0,1$ in the sum defining $d_P$ give
\[
 d_P\ge(e_Q-1)+(2-1)=e_Q,
\]
and consequently $\delta_Q\ge1$. These facts about ramification groups are recalled in \cite[Chapter~11]{HKT}.

We recall another standard property of a point stabilizer: $G_P^{(1)}$ is its unique Sylow $2$-subgroup and is therefore normal in $G_P$ \cite[Theorem 11.49]{HKT}. 

We shall also use the following bound for abelian automorphism groups:
\begin{equation}\label{eq:abelianbound}
 |A|\le4g+2
\end{equation}
for every abelian subgroup $A\le\Aut(\XX)$ in characteristic two; see \cite[Theorem~11.79]{HKT}.

In a least common multiple, a hat over an entry means that this entry is omitted. Thus
\[
 \lcm(e_1,\ldots,\widehat e_i,\ldots,e_v)
\]
is the least common multiple of all the $e_j$ with $j\ne i$.

\begin{lemma}\label{lem:kummer}
Let $K(\YY)/K(\ZZ)$ be a cyclic Kummer extension of odd degree $m$. It cannot ramify at exactly one point of $\ZZ$. Moreover, if $\ZZ$ is rational and the ramification indices are $e_1,\ldots,e_v$, then
\[
 \lcm(e_1,\ldots,e_v)=m,
 \qquad \lcm(e_1,\ldots,\widehat e_i,\ldots,e_v)=m
\]
for every $i$. In particular, if the extension ramifies over exactly two points, both ramification indices are equal to $m$ and $\YY$ is rational. If it ramifies over exactly three points, the integers $d_i=m/e_i$ are pairwise coprime.
\end{lemma}
\begin{proof}
The impossibility of ramification over exactly one point is standard; see \cite[Proposition~10]{DGTcyclic}. Its proof uses only that the extension is tame and therefore applies also in characteristic two.

Assume now that $\ZZ$ is rational. Fix one of the ramification points, say $Q_i$, and let $B_i$ be the subgroup of $\operatorname{Gal}(K(\YY)/K(\ZZ))$ generated by the stabilizers above all the other ramification points. Suppose that $B_i$ is a proper subgroup. Then the cover
\[
 \YY/B_i\longrightarrow\ZZ\cong\PP^1
\]
has degree $d>1$ and can ramify only over $Q_i$. Since it is tame, the total different over $Q_i$ is at most $d-1$. Riemann--Hurwitz would then give
\[
 2g(\YY/B_i)-2\le-2d+(d-1)=-d-1<-2,
\]
which is impossible. Hence $B_i$ is the whole cyclic group.

In a cyclic group of order $m$, subgroups of orders $e_j$, $j\ne i$, generate the whole group if and only if their orders have least common multiple $m$. Thus
\[
 \lcm(e_1,\ldots,\widehat e_i,\ldots,e_v)=m.
\]
The first least-common-multiple equality follows as well. If $v=2$, both indices are $m$, and Riemann--Hurwitz gives
\[
 2g(\YY)-2=-2m+2(m-1)=-2.
\]
Thus $\YY$ is rational. Finally, let $v=3$. If a prime divided both $d_i=m/e_i$ and $d_j=m/e_j$, then $\lcm(e_i,e_j)<m$, contrary to the equality obtained by omitting the third index. Hence the $d_i$ are pairwise coprime.
\end{proof}

\begin{lemma}\label{lem:P1involution}
Let $\tau$ be a nontrivial involution of a rational curve in characteristic two. Then $\tau$ has exactly one fixed point. After a suitable choice of a coordinate $x$, it has the form
\[
 \tau(x)=x+1.
\]
\end{lemma}
This is the standard description of involutions in $\operatorname{PGL}(2,K)$ in characteristic two; see \cite[Theorem~11.91]{HKT}.

\begin{lemma}\label{lem:P1cyclic}
Let $C$ be a cyclic group of odd order $n>1$ acting on a rational curve. Then $C$ fixes exactly two points. After choosing them as $0$ and $\infty$, a generator of $C$ has the form $x\mapsto\zeta_nx$. If an involution $t$ normalizes $C$ and satisfies $t\rho t=\rho^{-1}$ for every $\rho\in C$, then it exchanges $0$ and $\infty$ and has the form $x\mapsto c/x$ for some $c\in K^*$.
\end{lemma}
This is the standard form of cyclic and dihedral subgroups of $\operatorname{PGL}(2,K)$; see again \cite[Theorem~11.91]{HKT}.

We shall also use the standard uniqueness property of the hyperelliptic involution. If a curve of genus at least two admits a degree-two map onto a rational curve, the corresponding involution is unique and is therefore central in the full automorphism group. Consequently every automorphism leaves its fixed field invariant and induces an automorphism of the rational quotient; the kernel of the induced action is generated by the hyperelliptic involution. See, for example, \cite[Chapter~11]{HKT}.

We shall use the following consequence of Dickson's classification only when determining the full automorphism groups of the extremal dihedral curves.

\begin{lemma}\label{lem:Dickson}
Let $A$ be a finite subgroup of $\operatorname{PGL}(2,K)$ which contains a dihedral group $D_{2m}$, where $m\ge3$ is odd. Then one of the following holds:
\begin{enumerate}
\item $A$ is dihedral of order $2n$, with $n$ odd;
\item $A=E\rtimes C_n$, where $E$ is a nontrivial elementary abelian $2$-group and $n$ is odd;
\item $A\cong\operatorname{PSL}(2,q)=\operatorname{PGL}(2,q)$ for some $q=2^e$.
\end{enumerate}
An elementary abelian $2$-subgroup of $\operatorname{PGL}(2,K)$ has one common fixed point and acts freely outside it. Moreover, in the standard action of $\operatorname{PSL}(2,q)$ on $\PP^1(K)$, the orbits have lengths
\[
 q+1,\qquad q^2-q,\qquad q^3-q.
\]
The first orbit is $\PP^1(\FF_q)$; the last length occurs for every point outside $\PP^1(\FF_{q^2})$.
\end{lemma}
These statements are part of the standard classification of finite subgroups of $\operatorname{PGL}(2,K)$; see \cite[Theorem~11.91]{HKT}. Notice that the last two orbit lengths are even, whereas $q+1$ is odd.

\begin{lemma}\label{lem:kummernormalform}
Let $K(\YY)/K(x)$ be a cyclic Kummer extension of odd degree $n$. Assume that it ramifies only over $a_1,\ldots,a_v\in\PP^1$. Then it can be written in the form
\[
 y^n=f(x),
\]
where the zeros and poles of $f$ are contained in $\{a_1,\ldots,a_v\}$. Let $c_i=\ord_{a_i}(f)$, where $\ord_{a_i}(f)$ denotes the order of the zero of $f$ at $a_i$, or the negative of the order of its pole. Then
\[
 e_i=\frac{n}{\gcd(n,c_i)}
\]
is the ramification index over $a_i$, and $\sum_i c_i\equiv0\pmod n$. Replacing $y$ by $y^u$ with $\gcd(u,n)=1$ multiplies all the residues $c_i\pmod n$ by $u$.

The residues $c_i$ determine $f$ up to multiplication by an $n$th power in $K(x)$ and by a nonzero constant. More precisely, if two rational functions have the same orders modulo $n$ at every point of $\PP^1$, their quotient is $c h(x)^n$ for some $c\in K^*$ and $h(x)\in K(x)$.

Let $\rho$ be given by $\rho(y)=\zeta_ny$. If an automorphism $\tau$ of $\YY$ leaves $K(x)$ invariant and satisfies $\rho\tau=\tau\rho^a$, then
\[
 \tau(y)=h(x)y^a
\]
for some $h(x)\in K(x)$, where the exponent is read modulo $n$. In particular, if $\tau$ centralizes $\rho$, then $\tau(y)=h(x)y$.
\end{lemma}
\begin{proof}
The assertions about the equation and the ramification indices are standard facts from Kummer theory; see \cite[Proposition~3.7.3]{Stichtenoth}. We justify the two additional statements used later.

Suppose that $f_1$ and $f_2$ have the same orders modulo $n$ at every point. Then
\[
 (f_1/f_2)=nD
\]
for a divisor $D$ of degree zero on $\PP^1$. Every divisor of degree zero on $\PP^1$ is principal, so $D=(h)$ for some $h\in K(x)$. Therefore $f_1/f_2=ch^n$ for a constant $c\in K^*$.

Finally, write
\[
 \tau(y)=h_0(x)+h_1(x)y+\cdots+h_{n-1}(x)y^{n-1}.
\]
The relation $\rho\tau=\tau\rho^a$ gives
\[
 \sum_{i=0}^{n-1}\zeta_n^i h_i(x)y^i
 =\sum_{i=0}^{n-1}\zeta_n^a h_i(x)y^i.
\]
The representation in the basis $1,y,\ldots,y^{n-1}$ is unique. Hence $h_i=0$ for every $i\not\equiv a\pmod n$, and therefore $\tau(y)=h(x)y^a$.
\end{proof}

\begin{lemma}\label{lem:rational}
If $G$ is cyclic of order $N>2g-2$, then $\XX/G$ is rational. For the points $Q\in\XX/G$ over which the cover ramifies, the quantities $\delta_Q$ satisfy
\begin{equation}\label{eq:budget}
 \sum_Q\delta_Q=2+\frac{2g-2}{N}<3.
\end{equation}
Moreover, the stabilizer of some point contains the full Sylow $2$-subgroup of $G$.
\end{lemma}
\begin{proof}
Set $h=g(\XX/G)$. Since $N>2g-2$, we have $0<(2g-2)/N<1$. An unramified cover would give $(2g-2)/N=2h-2$, an integer, which is impossible. Thus the cover is ramified. By \eqref{eq:normalized},
\[
 0<\frac{2g-2}{N}=2h-2+\sum_Q\delta_Q<1.
\]
If $h\ge2$, the term $2h-2$ is at least two, which is impossible. Suppose that $h=1$. Then
\[
 0<\sum_Q\delta_Q<1.
\]
If the cover is wildly ramified at a point $P$, then $\delta_Q\ge1$ for $Q=\Phi(P)$, which is impossible. Hence the cover would have to be tame.

Write $N=2^a m$, with $m$ odd, and let $S\le G$ be the subgroup of order $2^a$. Put
\[
 \YY=\XX/S,\qquad \ZZ=\XX/G.
\]
We then have a tower
\[
 \XX\longrightarrow\YY\longrightarrow\ZZ,
\]
and $\YY\to\ZZ$ is cyclic of odd degree $m$.

If $m=1$, then $G=S$ is a $2$-group. Since $\XX\to\ZZ$ is tame, every point stabilizer is trivial. Thus the cover is unramified, a contradiction with what was proved above.

Assume $m>1$. The cover $\XX\to\ZZ$ is ramified and tame. Hence some point stabilizer has odd order greater than one, and its image gives a ramification point of $\YY\to\ZZ$. Thus $\YY\to\ZZ$ is ramified. By Lemma~\ref{lem:kummer}, it ramifies over at least two points. At each of them the ramification index is odd and at least three. Therefore these two points contribute at least
\[
 2\left(1-\frac13\right)=\frac43>1.
\]
This is again impossible. Therefore $h=0$, and \eqref{eq:budget} holds.

Let $B$ be the subgroup of $G$ generated by all the point stabilizers $G_P$. The cover $\XX/B\to\XX/G$ is unramified, because the image of every stabilizer in $G/B$ is trivial. If $B\ne G$, then $[G:B]\ge2$. Since $\XX/G$ is rational, Riemann--Hurwitz for this unramified cover gives
\[
 2g(\XX/B)-2=-2[G:B]\le-4,
\]
which is impossible. Therefore $B=G$. As $G$ is cyclic, its subgroups of $2$-power order form a chain. One of the stabilizers must consequently contain the entire Sylow $2$-subgroup of $G$.
\end{proof}

\begin{remark}\label{rem:ASreduction}
We shall use the reduced form of an Artin--Schreier equation; see \cite[Proposition~3.7.8]{Stichtenoth}. Starting from $y^2+y=f(z)$, replace $y$ by $y+h(z)$. The right-hand side becomes $f(z)+h(z)^2+h(z)$. If $f$ has a pole of even order $2r$, one can choose the leading term of $h$ so that the term of order $2r$ cancels. Repeating this operation gives an equation in which all pole orders are odd. If $Q$ is a pole of order $j$ in this reduced equation, then the different exponent at the point above $Q$ is $j+1$.
\end{remark}

Two equations $y^2+y=f(x)$ and $y^2+y=f'(x)$ define the same quadratic extension of $K(x)$ if and only if $f(x)+f'(x)=h(x)^2+h(x)$ for some $h(x)\in K(x)$. This gives the following standard lifting criterion.
\begin{lemma}\label{lem:ASlift}
Let $\XX$ be given by an Artin--Schreier equation
\[
 y^2+y=f(x),\qquad f(x)\in K(x),
\]
and let $\alpha$ be an automorphism of $K(x)$. Then $\alpha$ lifts to an automorphism of $K(\XX)$ if and only if there exists $h(x)\in K(x)$ such that
\begin{equation}\label{eq:ASlift}
 f(\alpha(x))+f(x)=h(x)^2+h(x).
\end{equation}
If \eqref{eq:ASlift} holds, the two lifts are
\[
 (x,y)\longmapsto(\alpha(x),y+h(x)),
 \qquad
 (x,y)\longmapsto(\alpha(x),y+h(x)+1).
\]
Moreover, every pole of a function of the form $h^2+h$ has even order.
\end{lemma}

We record the two standard estimates for higher ramification groups that will be used below.

\begin{lemma}\label{lem:breakestimates}
Suppose first that the first ramification group of a cyclic point stabilizer is $C_4$. Write its higher ramification groups in the form
\[
 C_4=G_P^{(1)}=\cdots=G_P^{(j_1)},
\]
\[
 G_P^{(j_1+1)}=\cdots=G_P^{(j_2)}=C_2,
 \qquad G_P^{(j_2+1)}=\{1\}.
\]
Thus $j_1$ is the largest integer $i$ for which $G_P^{(i)}$ has order four, and $j_2$ is the largest integer $i$ for which $G_P^{(i)}$ is nontrivial. Then
\[
 j_1\ge1,\qquad j_2\ge3j_1.
\]

Suppose next that a cyclic group of order eight is the point stabilizer. Define $r_1,r_2,r_3$ in the same way: its higher ramification groups have order eight up to $r_1$, order four from $r_1+1$ to $r_2$, order two from $r_2+1$ to $r_3$, and are trivial afterwards. Then
\begin{equation}\label{eq:breakbounds}
 r_1\ge1,\quad r_2\ge3r_1,\quad r_3\ge3r_2+2r_1.
\end{equation}
\end{lemma}
These inequalities are standard consequences of the ramification theory of cyclic extensions; see \cite{Schmid}, \cite[Lemma~2.2 and Theorem~2.3]{EK}, and \cite[Chapter~IV]{Serre}.

We also recall a standard local fact for tame cyclic groups. If a cyclic group $T=\langle\sigma\rangle$ of odd order $t$ fixes a point $P$, then one can choose a local parameter $v$ at $P$ such that $\sigma(v)=\zeta_t v$, where $\zeta_t$ is a primitive $t$th root of unity. This is the usual linear form of a tame cyclic action at a fixed point; see \cite[Chapter~IV]{Serre}.

We shall also use transitivity of the different; see \cite[Chapter~III]{Serre}. Let
\[
 \XX\longrightarrow\YY\longrightarrow\ZZ
\]
be a tower of covers, and let $P\mapsto Q\mapsto R$ be corresponding points. Transitivity gives
\[
 d(P\mid R)=d(P\mid Q)+e(P\mid Q)d(Q\mid R).
\]
Since $e(P\mid R)=e(P\mid Q)e(Q\mid R)$, division by $e(P\mid R)$ gives
\begin{equation}\label{eq:towerdifferent}
 \frac{d(P\mid R)}{e(P\mid R)}
 =\frac{d(Q\mid R)}{e(Q\mid R)}
  +\frac{d(P\mid Q)}{e(P\mid R)}
 \ge\frac{d(Q\mid R)}{e(Q\mid R)}.
\end{equation}
The following lemma summarizes the estimates needed for cyclic stabilizers.

\begin{lemma}\label{lem:local}
Suppose that a point stabilizer $G_P$ is cyclic of order $e=2^bt$, where $t$ is odd, and let $Q$ be the image of $P$ on $\XX/G_P$.
\begin{enumerate}
\item If $b=1$, there is an odd integer $j\ge t$, divisible by $t$, such that $G_P^{(i)}=C_2$ for $1\le i\le j$ and $G_P^{(i)}=\{1\}$ for $i>j$. In particular,
\begin{equation}\label{eq:order2different}
 \delta_Q=\frac{(2t-1)+j}{2t}
 =1-\frac1{2t}+\frac j{2t}
 \ge\frac32-\frac1{2t}.
\end{equation}
Let $T\le G_P$ be the subgroup of order $t$. The quadratic cover $\XX/T\to\XX/G_P$ has, in reduced Artin--Schreier form, a pole of order $j/t$ at $Q$.
\item If $b=2$, then $\delta_Q\ge9/4-1/(4t)\ge2$.
\item If $b\ge3$, then $\delta_Q\ge7/2$.
\end{enumerate}
\end{lemma}
\begin{proof}
First let $e=2t$, and write $G_P=\langle\tau\rangle\times T$, where $\tau$ has order two and $T=\langle\sigma\rangle$ has order $t$. For positive $i$, $G_P^{(i)}$ is a subgroup of $\langle\tau\rangle$. Since $P$ is wildly ramified, these groups are $\langle\tau\rangle$ for $1\le i\le j$ and trivial thereafter, for some positive integer $j$.

We verify the divisibility assertion directly. Choose a local parameter $v$ at $P$ with $\sigma(v)=\zeta_t v$, which is possible because $t$ is odd. By the definition of $G_P^{(i)}$, the first nonzero term of $\tau(v)-v$ has degree $j+1$; write
\[
 \tau(v)=v+c v^{j+1}+\text{terms of degree greater than }j+1,
 \qquad c\ne0.
\]
The equality $\tau\sigma=\sigma\tau$ gives, on comparing the coefficient of $v^{j+1}$, $\zeta_t c=\zeta_t^{j+1}c$. Therefore $t\mid j$. The quotient map $\XX\to\XX/\langle\tau\rangle$ has degree two. In reduced Artin--Schreier form, the pole order corresponding to $P$ is $j$; see \cite[Chapter~11]{HKT}. Hence $j$ is odd. Since $t\mid j$ and $j>0$, we have $j\ge t$. There are $j$ groups of order two after $G_P^{(0)}$, and therefore
\[
 d_P=(2t-1)+j.
\]
Division by $2t$ gives \eqref{eq:order2different}.

For the last assertion in (1), let $n$ be the pole order in the reduced Artin--Schreier equation for
\[
 \XX/T\longrightarrow\XX/G_P.
\]
The standard formula for the ramification break of a quotient gives, in this particular case,
\[
 n=\frac{1}{|G_P|}\sum_{i=1}^{j}|G_P^{(i)}|
   =\frac{1}{2t}\sum_{i=1}^{j}2
   =\frac jt;
\]
see \cite[Chapter~IV]{Serre}. For a quadratic Artin--Schreier extension this ramification break is precisely the pole order in a reduced equation. This proves the last assertion.

For (2), let $j_1$ be the largest integer for which $G_P^{(j_1)}$ has order four, and let $j_2$ be the largest integer for which $G_P^{(j_2)}$ has order two. The same coefficient comparison as above gives $t\mid j_1$, so $j_1\ge t$. Lemma~\ref{lem:breakestimates} gives $j_2\ge3j_1$. We now count the terms in the definition of the different. The term with $i=0$ is $4t-1$. There are $j_1$ further terms, for $1\le i\le j_1$, in which the ramification group has order four, and hence each contributes $4-1=3$. Finally, there are $j_2-j_1$ terms in which the group has order two, and each contributes one. Therefore
\[
 d_P=(4t-1)+3j_1+(j_2-j_1)
     =4t-1+2j_1+j_2\ge9t-1.
\]
Dividing by $4t$ proves (2).

For (3), let $T\le G_P$ be the subgroup of index eight. In the cover
\[
 \XX/T\longrightarrow\XX/G_P
\]
the stabilizer at the image of $P$ is cyclic of order eight. Let $r_1,r_2,r_3$ be the largest indices for which its ramification groups have orders eight, four and two, respectively. The inequalities \eqref{eq:breakbounds} give
\[
 r_1\ge1,\qquad r_2\ge3r_1,\qquad
 r_3\ge3r_2+2r_1.
\]
In particular, $r_2\ge3$ and $r_3\ge11$. We again count the contributions to the different. The term with $i=0$ is seven. The next $r_1$ terms contribute seven each; the following $r_2-r_1$ terms contribute three each; and the last $r_3-r_2$ nonzero terms contribute one each. Hence
\[
 7+7r_1+3(r_2-r_1)+(r_3-r_2)
 =7+4r_1+2r_2+r_3\ge28.
\]
The different exponent for this degree-eight cover is therefore at least $28$. Formula~\eqref{eq:towerdifferent}, applied to the tower obtained from $T$, gives
\[
 \delta_Q=\frac{d_P}{|G_P|}\ge\frac{28}{8}=\frac72.
\]
\end{proof}

\begin{lemma}\label{lem:arithmetic}
Let $m$ be odd, and let $d_0,d_1,d_2$ be pairwise coprime divisors of $m$, with $m/d_i\ge3$ for every $i$. Then
\begin{equation}\label{eq:arithmetic}
 m+2\ge d_0+2d_1+2d_2.
\end{equation}
\end{lemma}
\begin{proof}
We distinguish the possible cases according to which of the integers $d_i$ are equal to one.

Assume first that $d_1,d_2>1$. Since they are odd and coprime, their smallest possible values are three and five, in some order. Hence
\[
 (d_1-2)(d_2-2)\ge3,
\]
which is equivalent to
\[
 d_1d_2+1\ge2d_1+2d_2.
\]
The integers $d_0,d_1,d_2$ are pairwise coprime divisors of $m$, so their product divides $m$. Thus $m\ge d_0d_1d_2$, and
\[
 \begin{aligned}
 m+2-d_0-2d_1-2d_2
 &\ge d_0d_1d_2+2-d_0-2d_1-2d_2\\
 &\ge d_1d_2+1-2d_1-2d_2\ge0.
 \end{aligned}
\]

Assume next that $d_1=1<d_2$ and $d_0>1$. The numbers $d_0,d_2$ are odd and coprime, and $m\ge d_0d_2$. If $d_0=3$, then $d_2\ge5$, and
\[
 d_0d_2-d_0-2d_2=d_2-3\ge0.
\]
If $d_0\ge5$, then $d_2\ge3$, and
\[
 d_0d_2-d_0-2d_2
 =d_0(d_2-1)-2d_2
 \ge5(d_2-1)-2d_2=3d_2-5>0.
\]
In both cases $m+2\ge d_0+2+2d_2$. The case $d_2=1<d_1$ is obtained by exchanging $d_1$ and $d_2$.

Suppose that $d_0=d_1=1<d_2$. The inequality $m/d_2\ge3$ gives
\[
 m\ge3d_2\ge3+2d_2=d_0+2d_1+2d_2.
\]
The case $d_0=d_2=1<d_1$ is symmetric. Finally, if $d_1=d_2=1$, then $m\ge3d_0$. If $d_0>1$, this gives $m\ge d_0+4$; if $d_0=1$, the desired inequality is $m+2\ge5$, which follows from $m\ge3$. This proves the lemma.
\end{proof}

For a $2$-group $S$, the Deuring--Shafarevich formula reads
\begin{equation}\label{eq:DS}
 \gamma(\XX)-1=|S|(\gamma(\XX/S)-1)
                 +\sum_{\mathcal O\text{ short}}(|S|-|\mathcal O|).
\end{equation}
For an Artin--Schreier extension of degree two of a rational function field, with $r$ poles of odd orders $j_1,\ldots,j_r$, the formulas become
\begin{equation}\label{eq:ASgenus}
 2g(\XX)-2=-4+\sum_{i=1}^r(j_i+1),\qquad\gamma(\XX)=r-1.
\end{equation}
See \cite[Chapter 11]{HKT} and \cite{Nakajima}.

We shall use one further standard fact in Section~\ref{sec:hermitian}. If a curve that is maximal over $\FF_{q^2}$ admits a nonconstant $\FF_{q^2}$-morphism onto a curve $\mathcal C$, then $\mathcal C$ is also $\FF_{q^2}$-maximal. See \cite[Chapter~10]{HKT} and \cite{GSX}.

\section{Large cyclic groups}

Throughout this section, $G\le\Aut(\XX)$ is cyclic of order
\[
 N=2^a m,
\]
where $m$ is odd. We first exclude $a\ge2$. We then treat separately $a=0$ and $a=1$.

\begin{proposition}\label{prop:no4}
Under the hypotheses of Theorem~\ref{thm:cyclic}, $4\nmid N$.
\end{proposition}
\begin{proof}
Suppose that $a\ge2$. Lemma~\ref{lem:rational} gives a point $Q_0\in\XX/G$ such that the stabilizer of a point above $Q_0$ has order $2^a t$, for some divisor $t$ of $m$. If $a\ge3$, Lemma~\ref{lem:local}(3) gives $\delta_{Q_0}\ge7/2$, whereas \eqref{eq:budget} gives $\sum_Q\delta_Q<3$. This is impossible.

Suppose that $a=2$. Lemma~\ref{lem:local}(2) gives
\[
 \delta_{Q_0}\ge\frac94-\frac1{4t}\ge2.
\]
The cover is wildly ramified over $Q_0$. It cannot be wildly ramified over another point, since that point would contribute at least one and the sum $\sum_Q\delta_Q$ would be at least three, a contradiction. If the cover is tamely ramified over a point $Q\ne Q_0$, its ramification index is odd and at least three, so $Q$ contributes at least $2/3$. Hence there is at most one such point.

If $m=1$, then $N=4<2g+1$ because $g\ge2$, against the assumption. Suppose $m>1$, and set $S=C_4\le G$. The quotient $\XX/S\to\XX/G$ has cyclic group $G/S$ of odd order $m$. Let $P$ be a point of $\XX$ above $Q\in\XX/G$, and put $e_Q=|G_P|$. The stabilizer of the image of $P$ in $G/S$ is $G_PS/S$. Hence the ramification index of $\XX/S\to\XX/G$ over $Q$ is
\[
 e'_Q=|G_PS/S|=\frac{|G_P|}{|G_P\cap S|}
     =\frac{e_Q}{\gcd(e_Q,4)},
\]
where the last equality follows because $G$ is cyclic. At $Q_0$ this is $e'_{Q_0}=t$, since $e_{Q_0}=4t$. At every other point of $\XX/G$ over which $\XX\to\XX/G$ ramifies, the stabilizer has odd order, and therefore $e'_Q=e_Q$. We know that there is at most one such point. The Kummer cover $\XX/S\to\XX/G$ has degree $m>1$ and target curve $\XX/G\cong\PP^1$. It cannot be unramified, and Lemma~\ref{lem:kummer} excludes ramification over exactly one point. Consequently it ramifies over exactly two points, $Q_0$ and another point $Q_1$. The last assertion of Lemma~\ref{lem:kummer} for a two-point cover gives
\[
 e'_{Q_0}=e'_{Q_1}=m.
\]
Thus $t=m$, and the stabilizers for $\XX\to\XX/G$ above $Q_0,Q_1$ have orders $4m,m$, respectively. Using \eqref{eq:normalized} and Lemma~\ref{lem:local}(2) gives
\[
 \sum_Q\delta_Q\ge\left(\frac94-\frac1{4m}\right)
                         +\left(1-\frac1m\right).
\]
Indeed, multiplying the last inequality by $4m$ and subtracting $8m$ yields
\[
 2g-2\ge-8m+4m\left(\frac94-\frac1{4m}+1-\frac1m\right)
            =5m-5.
\]
Consequently $2g+1\ge5m-2$, whereas the hypothesis says $4m=N\ge2g+1$. These inequalities give $m\le2$, contradicting the assumption that $m>1$ is odd.
\end{proof}

\begin{proof}[Proof of Theorem~\ref{thm:cyclic} when $N$ is odd]
The cover $\XX\to\XX/G$ is tame, and Lemma~\ref{lem:rational} shows that $\XX/G$ is rational. If the cover ramified over only one point, this would contradict Lemma~\ref{lem:kummer}. If it ramified over exactly two points, the same lemma would give $g=0$. Therefore it ramifies over at least three points.

At each of these points the ramification index is an odd integer $e\ge3$. Its contribution to \eqref{eq:budget} is
\[
 \delta_Q=1-\frac1e\ge1-\frac13=\frac23.
\]
Since five such contributions would sum to at least $10/3>3$, the cover ramifies over at most four points.

Assume that the cover ramifies over four points, with indices $e_1\le e_2\le e_3\le e_4$. If $e_1\ge5$, the four contributions sum to at least $4(1-1/5)>3$. If $e_1=3$ and $e_2\ge5$, their sum is at least $2/3+3(4/5)>3$. If $e_1=e_2=3$ and $e_3\ge7$, the sum is at least $2(2/3)+2(6/7)>3$. Therefore $e_1=e_2=3$ and $e_3$ is either three or five.

If $e_3=3$, Lemma~\ref{lem:kummer} applied with $i=4$ shows that $N=\lcm(3,3,3)=3$. Then $e_4=3$. Riemann--Hurwitz gives
\[
 2g-2=-2\cdot3+4(3-1)=2,
\]
so $g=2$ and $N=3<2g+1$.
If $e_3=5$, Lemma~\ref{lem:kummer} applied with $i=4$ gives $N=\lcm(3,3,5)=15$. The fourth index is a divisor of $15$ and is at least five, hence $e_4=5$ or $15$. In the first case
\[
 2g-2=15\left(-2+\frac23+\frac23+\frac45+\frac45\right)=14;
\]
in the second the last $4/5$ becomes $14/15$ and $2g-2=16$. Thus $g=8$ or $9$, and in either case $N=15<2g+1$. Therefore the cover cannot ramify over four points.

There are exactly three points of $\XX/G$ over which the cover ramifies. Choose a coordinate $x$ on $\XX/G$ so that they are $0,1,\infty$. By Lemma~\ref{lem:kummernormalform}, after multiplying $y$ by a rational function and taking a constant $N$th root, a Kummer equation has the form
\[
 y^N=x^r(x+1)^s.
\]
The exponent at infinity is $-r-s$ modulo $N$. All three exponent residues are nonzero because the cover ramifies over all three points. Choose representatives between one and $N-1$. Their sum is divisible by $N$ and lies strictly between zero and $3N$, so it is equal to $N$ or $2N$. Replacing $y$ by $y^{-1}$ changes every residue to its negative and exchanges these two possibilities. We may therefore arrange $r+s<N$.

The extension has degree $N$ precisely when $\gcd(N,r,s)=1$. If a common divisor $d>1$ existed, the right-hand side would be a $d$th power and the Kummer equation would have degree at most $N/d$. Conversely, if the right-hand side were a proper power in $K(x)$, its valuations $r$ and $s$ at $0$ and $1$ would share a divisor with $N$. The ramification indices at the three ramification points are
\[
 \frac N{\gcd(N,r)},\qquad
 \frac N{\gcd(N,s)},\qquad
 \frac N{\gcd(N,r+s)}.
\]
Substituting their contributions $1-\gcd(N,r)/N$, $1-\gcd(N,s)/N$ and $1-\gcd(N,r+s)/N$ in \eqref{eq:normalized} gives
\[
 2g-2=-2N+3N-\gcd(N,r)-\gcd(N,s)-\gcd(N,r+s),
\]
which is \eqref{eq:tamegenus}. Conversely, each of the three greatest common divisors is at least one. Hence $2g\le N-1$, or equivalently $N\ge2g+1$. This proves case~\textup{(I)}.
\end{proof}

Assume now that $N$ is even. By Proposition~\ref{prop:no4}, $N=2m$, where $m\ge3$ is odd. Let $P=C_2$ and $B=C_m$ be the unique subgroups of $G$ of those orders. Put
\[
 \YY=\XX/P,\qquad\ZZ=\XX/G\cong\PP^1.
\]
The extension $K(\YY)/K(\ZZ)$ is cyclic Kummer of degree $m$.

\begin{lemma}\label{lem:onewild}
There is exactly one point $Q\in\ZZ$ such that $\XX\to\ZZ$ is wildly ramified over $Q$.
\end{lemma}
\begin{proof}
Lemma~\ref{lem:rational} gives at least one point of $\ZZ$ over which the cover is wildly ramified. Each such point contributes at least one to \eqref{eq:budget}, so there are at most two. Suppose that there are two, with stabilizers of orders $2t_1$ and $2t_2$.

Assume first that $\XX\to\ZZ$ ramifies over no other point. The Kummer cover $\YY\to\ZZ$ can then ramify only over these two points. It is a nontrivial tame cover of the rational curve $\ZZ$, so it cannot be unramified; by Lemma~\ref{lem:kummer}, it cannot ramify over exactly one point either. Hence it ramifies over both points, with indices $t_1,t_2$. Lemma~\ref{lem:kummer} now forces $t_1=t_2=m$. By \eqref{eq:order2different}, each contribution is at least $3/2-1/(2m)$. Substituting their sum into \eqref{eq:normalized} gives
\[
 2g-2\ge2m\left(-2+3-\frac1m\right)=2m-2.
\]
Hence $g\ge m$, whereas $2m\ge2g+1$ implies $g<m$.

Suppose that $\XX\to\ZZ$ is also tamely ramified over another point. If one $t_i>1$, then $t_i\ge3$ and the corresponding contribution is at least $3/2-1/6=4/3$. The second wildly ramified point contributes at least one, and the tamely ramified point contributes at least $2/3$. Their sum is at least three, against \eqref{eq:budget}.

It remains to consider $t_1=t_2=1$. The cover $\YY\to\ZZ$ is unramified over these two points. Since it has degree $m>1$, it must ramify over at least two further points by Lemma~\ref{lem:kummer}. They contribute at least $2(2/3)$. Together with the two wildly ramified points this gives more than three, again a contradiction.
\end{proof}

\begin{proof}[Completion of the proof of Theorem~\ref{thm:cyclic}]
Let $Q_0\in\ZZ$ be the unique point over which $\XX\to\ZZ$ is wildly ramified. Write $2t$ for the order of the stabilizer above $Q_0$, and let $j$ be the integer in Lemma~\ref{lem:local}(1). Thus $j$ is an odd multiple of $t$ and $j\ge t$. The point $Q_0$ contributes at least one. Every point over which the cover is tamely ramified contributes at least $2/3$. Therefore there are at most two such points.

There is at least one. Indeed, if $t>1$, the Kummer cover $\YY\to\ZZ$ ramifies over $Q_0$ and must ramify over another point. If $t=1$, it is unramified over $Q_0$ and must ramify over at least two points.

Suppose first that $\XX\to\ZZ$ is tamely ramified over exactly one point $Q_1$. If $t=1$, then $\YY\to\ZZ$ is unramified over $Q_0$ and can ramify only over $Q_1$, which is impossible by Lemma~\ref{lem:kummer}. Thus $t>1$, and $\YY\to\ZZ$ ramifies over both $Q_0$ and $Q_1$. Lemma~\ref{lem:kummer} gives ramification index $m$ at both points. Hence $t=m$, and the stabilizers for $\XX\to\ZZ$ have orders $2m$ and $m$. Their different exponents are $2m-1+j$ and $m-1$. The fiber over $Q_0$ has one point, and the fiber over $Q_1$ has two. Thus \eqref{eq:RH} gives
\begin{equation}\label{eq:twobranchwild}
 2g-2=-4m+(2m-1+j)+2(m-1)=j-3.
\end{equation}
The equality $2g-2=j-3$ is equivalent to $2g+1=j$. The hypothesis $N=2m\ge2g+1$ gives $j\le2m$. But $j$ is a positive odd multiple of $m$; the only possibility is $j=m$. Consequently $g=(m-1)/2$.

Suppose now that $\XX\to\ZZ$ is tamely ramified over two points $Q_1,Q_2$, with indices $e_1,e_2$. Assume first that $t=1$. Then $\YY\to\ZZ$ is unramified over $Q_0$ and ramifies over $Q_1,Q_2$. Both ramification indices are $m$. The stabilizers for $\XX\to\ZZ$ therefore have orders $2,m,m$. The fiber over $Q_0$ has $m$ points, each with different exponent $j+1$. Each of the other two fibers has two points, each with different exponent $m-1$. Hence
\begin{equation}\label{eq:threebranchwild}
 2g-2=2m-4+m(j-1).
\end{equation}
Indeed, $2g+1=2m-1+m(j-1)$. The inequality $2g+1\le2m$ therefore yields $m(j-1)\le1$. Since $m\ge3$ and $j$ is a positive odd integer, we obtain $j=1$ and $g=m-1$.

It remains to treat $t>1$. The Kummer cover $\YY\to\ZZ$ has three ramification points of indices $t,e_1,e_2$. Put $d_0=m/t$, $d_1=m/e_1$, and $d_2=m/e_2$. By Lemma~\ref{lem:kummer}, these three integers are pairwise coprime. Moreover,
\[
 \frac{m}{d_0}=t\ge3,\qquad
 \frac{m}{d_i}=e_i\ge3\quad(i=1,2),
\]
so Lemma~\ref{lem:arithmetic} applies.

Above $Q_0$ the stabilizer has order $2t$, the different exponent is $2t-1+j$, and the fiber has $m/t=d_0$ points. Above $Q_i$, for $i=1,2$, the fiber has $2m/e_i=2d_i$ points, each with different exponent $e_i-1$. Its total contribution to \eqref{eq:RH} is
\[
 2d_i(e_i-1)=2m-2d_i.
\]
Consequently
\[
 \begin{aligned}
 2g-2
 &=-4m+d_0(2t-1+j)+(2m-2d_1)+(2m-2d_2)\\
 &=2m-2d_1-2d_2+d_0(j-1)\\
 &\ge2m-2d_1-2d_2+m-d_0\\
 &\ge2m-2,
 \end{aligned}
\]
because $j\ge t$, so $d_0j\ge m$, and Lemma~\ref{lem:arithmetic} gives $m+2\ge d_0+2d_1+2d_2$. This implies $g\ge m$, contrary to $2m\ge2g+1$.

In both remaining cases $\YY\to\ZZ$ ramifies over exactly two points, with ramification index $m$ at each. Lemma~\ref{lem:kummer} gives $\YY\cong\PP^1$. Choose coordinates $x,z$ such that $K(\YY)=K(x)$, $K(\ZZ)=K(z)$, and $z=x^m$. The field $K(\XX)^B$, where $B=C_m$, has degree two over $K(z)$. This quadratic extension ramifies over a point $Q\in\ZZ$ precisely when the stabilizer above $Q$ for $\XX\to\ZZ$ has even order. By Lemma~\ref{lem:onewild}, this happens only at $Q_0$. Hence its Artin--Schreier equation has exactly one pole, at $Q_0$. By Lemma~\ref{lem:local}(1), its order is $j/t=1$.

In the first case $\YY\to\ZZ$ ramifies over $Q_0$. Put $Q_0$ at $z=\infty$. A function in $K(z)$ with no pole except a simple pole at infinity has the form $az+b$. Since $K$ is algebraically closed, the constant $b$ can be removed by an Artin--Schreier change of variable. Hence
\[
 y^2+y=az,\qquad a\ne0.
\]
Substituting $z=x^m$ and scaling $x$ gives \eqref{eq:onepole}.

In the second case $\YY\to\ZZ$ is unramified over $Q_0$. Put its two ramification points at $z=0,\infty$ and put $Q_0$ at $z=1$. A function in $K(z)$ with only a simple pole at $z=1$ has the form $c/(z+1)+b$. Removing the constant gives
\[
 y^2+y=\frac{c}{z+1},\qquad c\ne0.
\]
Substitution $z=x^m$ yields \eqref{eq:ordinary}.

The extensions $K(\YY)/K(z)$ and $K(\XX)^B/K(z)$ have coprime degrees $m$ and two. Their intersection is therefore $K(z)$, and their compositum has degree $2m$. Since $[K(\XX):K(z)]=2m$, this compositum is $K(\XX)$. In both displayed equations, the automorphisms $x\mapsto\zeta_m x$ and $y\mapsto y+1$ commute. Their product has order $2m$ and generates the original group.

In case~\textup{(II)}, the function $x^m\in K(x)$ has one pole of odd order $m$. Formula \eqref{eq:ASgenus} gives $2g-2=-4+(m+1)=m-3$ and $\gamma(\XX)=0$. In case~\textup{(III)}, the polynomial $x^m+1$ has $m$ distinct roots, and $c/(x^m+1)$ has a simple pole at each of them. Formula \eqref{eq:ASgenus} gives $2g-2=-4+2m$ and $\gamma(\XX)=m-1=g$. Finally, the respective cyclic orders are $2m=4g+2$ and $2m=2g+2$, both at least $2g+1$. This proves the converse and Theorem~\ref{thm:cyclic}.
\end{proof}

\begin{remark}
The theorem also shows that the bound \eqref{eq:abelianbound} is attained: the curve $y^2+y=x^{2g+1}$ has a cyclic group of order $4g+2$.
\end{remark}

\section{Groups with a cyclic subgroup of index two}
The aim of this section is to classify the curves admitting a group $H\le\Aut(\XX)$ with a cyclic subgroup $G$ of index two and $|H|>4g+4$. We first determine whether the cyclic groups of even order in Theorem~\ref{thm:cyclic} can be contained in such a group.

For the curves in cases~\textup{(II)} and~\textup{(III)} of Theorem~\ref{thm:cyclic}, put
\[
 \iota:(x,y)\mapsto(x,y+1),\qquad
 \varphi:(x,y)\mapsto(\zeta_mx,y).
\]
Then $G=\langle\iota,\varphi\rangle$ is cyclic of order $2m$.

\begin{proposition}\label{prop:normalizers}
For $g\ge2$, the normalizer of $G$ in the full automorphism group is
\begin{equation}\label{eq:normalizers}
 \NN_{\Aut(\XX)}(G)=
 \begin{cases}
 G,&y^2+y=x^m,\\
 D_{4m},&y^2+y=c/(x^m+1).
 \end{cases}
\end{equation}
In the second case, the normalizer is generated by $G$ and
\begin{equation}\label{eq:reflection}
 t:(x,y)\mapsto(x^{-1},y+d),\qquad d^2+d=c.
\end{equation}
\end{proposition}
\begin{proof}
The fixed field of $\iota$ is the rational function field $K(x)$, so $\iota$ is the hyperelliptic involution. For a curve of genus at least two the hyperelliptic involution is unique; see \cite[Chapter~11]{HKT}. Consequently $\iota$ is central in $\Aut(\XX)$.

Since $\iota$ is central, every automorphism of $\XX$ leaves $K(x)$ invariant and induces an automorphism of this field. If it normalizes $G$, the induced automorphism normalizes $\langle\varphi\rangle\cong C_m$. This cyclic group fixes exactly the places $x=0$ and $x=\infty$. Hence the induced automorphism preserves or exchanges them. So, it has the form $x\mapsto ax$ or $x\mapsto a/x$, for some $a\in K^*$.

For $y^2+y=x^m$, the extension $K(\XX)/K(x)$ ramifies only over the infinite place of $K(x)$. An automorphism that lifts must preserve this place. Hence $x\mapsto a/x$ cannot lift. Consider $x\mapsto ax$. By Lemma~\ref{lem:ASlift}, it lifts if and only if there is $h\in K(x)$ such that
\[
 h^2+h=(a^m+1)x^m.
\]
If $a^m+1\ne0$, the right-hand side has a pole of odd order $m$ at infinity. This is impossible because every pole of $h^2+h$ has even order. Therefore $a^m=1$. Both lifts of these $m$ dilations are already in $G$, proving the first case.

For $y^2+y=c/(x^m+1)$, the extension $K(\XX)/K(x)$ ramifies over the $m$ places determined by $x^m=1$. A map $x\mapsto ax$ or $x\mapsto a/x$ preserves this set only if $a^m=1$. Every such dilation $x\mapsto ax$ lifts by leaving $y$ fixed. The map $x\mapsto x^{-1}$ also lifts, because in characteristic two
\[
 \frac{c}{x^{-m}+1}+\frac{c}{x^m+1}=c.
\]
Choose $d$ as in \eqref{eq:reflection}. Lemma~\ref{lem:ASlift} shows that $t$ lifts $x\mapsto x^{-1}$. A direct calculation gives $t^2=1$. Moreover, $t$ centralizes $\iota$ and satisfies $t\varphi t=\varphi^{-1}$. The group generated by $G$ and $t$ has order $4m$ and contains all possible lifts, because the kernel of the restriction to $K(x)$ is $\langle\iota\rangle$. If $\rho:(x,y)\mapsto(\zeta_mx,y+1)$, then $t\rho t=\rho^{-1}$, so this group is $D_{4m}$.
\end{proof}

\begin{corollary}\label{cor:oddG}
Under the hypotheses of Theorem~\ref{thm:index}, the cyclic subgroup $G$ has odd order.
\end{corollary}
\begin{proof}
Put $N=|G|=|H|/2$. From $|H|>4g+4$ we obtain $N>2g+2$. Suppose that $N$ is even. By Theorem~\ref{thm:cyclic}, $\XX$ belongs to case~\textup{(II)} or~\textup{(III)}. In case~\textup{(III)}, $N=2g+2$, contrary to $N>2g+2$. In case~\textup{(II)}, Proposition~\ref{prop:normalizers} gives $\NN_{\Aut(\XX)}(G)=G$. Since an index-two subgroup is normal, we have $H\le\NN_{\Aut(\XX)}(G)=G$, again a contradiction. Thus $N$ is odd. Finally, $H$ cannot be cyclic, because its order would contradict \eqref{eq:abelianbound}.
\end{proof}

\begin{proof}[Proof of Theorem~\ref{thm:index}]
Assume the hypotheses of Theorem~\ref{thm:index}. By Corollary~\ref{cor:oddG}, $G=C_N$ with $N$ odd. The group $H$ has order $2N$; by Sylow's theorem it contains an involution $t$ outside $G$. Since $G$ is normal of index two, $H=G\rtimes\langle t\rangle$. If $\rho$ generates $G$, conjugation by $t$ sends $\rho$ to $\rho^a$ for some integer $a$ coprime to $N$. Applying $t$ twice gives $a^2\equiv1\pmod N$, and thus
\begin{equation}\label{eq:semidirect}
 H=\langle\rho,t\mid\rho^N=t^2=1,\ t\rho t=\rho^a\rangle,
 \qquad a^2\equiv1\pmod N.
\end{equation}
The induced involution on $\XX/G\cong\PP^1$ is nontrivial. Indeed, an automorphism acting trivially on $K(\XX)^G$ belongs to the Galois group $G$ of $K(\XX)/K(\XX)^G$, while $t\notin G$.

By Theorem~\ref{thm:cyclic}, there are exactly three points of $\XX/G$ over which the $G$-cover ramifies. The induced involution permutes them. By Lemma~\ref{lem:P1involution}, it has exactly one fixed point on $\XX/G$. Hence it cannot fix all three points: it fixes one of them and exchanges the other two. Choose the fixed point as infinity and the exchanged points as $0,1$. The last assertion of Lemma~\ref{lem:P1involution} allows us to choose the coordinate so that the involution is $x\mapsto x+1$.

Write $y^N=x^r(x+1)^s$. Since $t$ exchanges $0$ and $1$, the corresponding stabilizers have the same order. Their indices are $N/\gcd(N,r)$ and $N/\gcd(N,s)$, so $\gcd(N,r)=\gcd(N,s)$. A common divisor also divides $\gcd(N,r,s)=1$. Both greatest common divisors therefore equal one, and the two points are totally ramified. Since $r$ is invertible modulo $N$, choose $u$ with $ur\equiv1\pmod N$, replace $y$ by $y^u$, and multiply the result by an appropriate rational function. The new residue at $x=1$ is $s'\equiv us\pmod N$. We choose $1\le s'<N$ and then rename $s'$ as $s$. Thus the two residues are $1,s$, with $\gcd(s,N)=1$.
Thus, from now on, we use the equation
\[
 y^N=x(x+1)^s.
\]

Let $\rho$ act by $\rho(y)=\zeta_Ny$. From $t\rho t=\rho^a$ and $t^2=1$ we obtain
\[
 \rho t=t\rho^a.
\]
Lemma~\ref{lem:kummernormalform} now gives
\[
 t(y)=h(x)y^a
\]
for some $h(x)\in K(x)$. The map $t$ exchanges $x=0$ and $x=1$. At these two points the exponent residues of the Kummer equation are $1$ and $s$, respectively. Therefore multiplication by $a$ must send the first residue to the second one, and hence $a\equiv s\pmod N$. Since $a^2\equiv1\pmod N$, we obtain
\begin{equation}\label{eq:exponent}
 s\equiv a\pmod N,\qquad s^2\equiv1\pmod N.
\end{equation}
The congruence $s^2\equiv1\pmod N$ makes $\ell=(s^2-1)/N$ an integer. The following map is a lift of $x\mapsto x+1$:
\begin{equation}\label{eq:Kummerlift}
 t_0:(x,y)\mapsto\left(x+1,\frac{y^s}{(x+1)^\ell}\right),
 \qquad \ell=\frac{s^2-1}{N}.
\end{equation}
Indeed, its $N$th power in the second coordinate is
\[
 \frac{(x(x+1)^s)^s}{(x+1)^{\ell N}}
 =x^s(x+1)^{s^2-\ell N}=x^s(x+1),
\]
which is the defining equation after $x$ is replaced by $x+1$. Applying the map twice to $y$ gives
\[
 \frac{y^{s^2}}{(x+1)^{\ell s}x^\ell}
 =y\,\frac{(x(x+1)^s)^\ell}{(x+1)^{\ell s}x^\ell}=y.
\]
Thus it is indeed an involution.

The original involution $t$ and $t_0$ induce the same automorphism of $K(x)$. Therefore $t_0t^{-1}$ belongs to the Galois group $G$. In particular $t_0\in H$, and we may replace $t$ by $t_0$. From now on we again denote this involution by $t$.

Put $M=\gcd(N,s+1)$ and $k=N/M$. Since the first two ramification points are totally ramified, their three greatest common divisors in \eqref{eq:tamegenus} are $1,1,M$, respectively. Therefore
\begin{equation}\label{eq:NminusM}
 2g=N-M=M(k-1).
\end{equation}
The numbers $s-1$ and $s+1$ differ by two, and $N$ is odd; hence no prime dividing $N$ divides both. As $N$ divides $(s-1)(s+1)$, each prime power in $N$ must divide exactly one of these two factors. By definition, $M$ is the product of those prime powers dividing $s+1$, and $k=N/M$ is the product of the others. It follows that
\begin{equation}\label{eq:CRT}
 \gcd(M,k)=1,\qquad s\equiv-1\pmod M,\qquad s\equiv1\pmod k.
\end{equation}
Both $M$ and $k$ are odd. Since $2g=N-M$, the inequality $N>2g+2$ becomes $N>N-M+2$, or $M>2$. Also $g=M(k-1)/2\ge2$ forces $k>1$. Thus $M,k\ge3$.

Decompose $G=C_k\times C_M$. Relations \eqref{eq:semidirect} and \eqref{eq:CRT} say that $t$ centralizes $C_k$ and sends every element of $C_M$ to its inverse; in particular, if $b$ generates $C_M$, then $tbt=b^{-1}$. Therefore
\begin{equation}\label{eq:directproduct}
 H\cong C_k\times D_{2M}.
\end{equation}
This gives the group-theoretic part of Theorem~\ref{thm:index}.

Let $L=C_k\le G$ and $\YY=\XX/L$. Above each of $0,1$, the $G$-cover has one point, fixed by all of $G$, hence also by $L$. At infinity the $G$-stabilizer has order $N/M=k$ by \eqref{eq:NminusM}; it therefore equals $L$. The fiber has $N/k=M$ points, all fixed by $L$. No other point has a nontrivial stabilizer in $L$, because the $G$-cover has only these three ramification points. The cover $\XX\to\YY$ is tame of degree $k$, totally ramified at $M+2$ points. Its different is $(M+2)(k-1)$, and Riemann--Hurwitz gives
\[
 M(k-1)-2=k(2g(\YY)-2)+(M+2)(k-1).
\]
The left side equals $M(k-1)-2$ by \eqref{eq:NminusM}; subtracting $(M+2)(k-1)$ from both sides gives $-2k=k(2g(\YY)-2)$. Thus $g(\YY)=0$.

The group $H/L\cong D_{2M}$ acts faithfully on this rational curve. By Lemma~\ref{lem:P1cyclic}, we may choose a coordinate $u$ such that
\[
 b:u\mapsto\zeta_Mu,\qquad t:u\mapsto c/u
\]
for some $c\in K^*$. Since $K$ is algebraically closed, we can replace $u$ by a constant multiple and obtain $c=1$. Thus
\[
 b:u\mapsto\zeta_Mu,\qquad t:u\mapsto u^{-1}.
\]

The $M+2$ ramification points of $\XX\to\YY$ lie over the three points of $\XX/G$ considered above. Two of them are fixed by $C_M$, and the remaining $M$ points form one $C_M$-orbit. Consequently they are
\[
 0,\quad\infty,\quad\text{and}\quad \{u:u^M=\lambda\}
\]
for some $\lambda\in K^*$. The map $u\mapsto u^{-1}$ sends the last set to $\{u:u^M=\lambda^{-1}\}$. Since it preserves the ramification points, $\lambda=\lambda^{-1}$. Therefore $\lambda^2=1$, and in characteristic two this implies $\lambda=1$. Hence the set of ramification points is
\begin{equation}\label{eq:branchset}
 \{0,\infty\}\cup\{u:u^M=1\}.
\end{equation}

The cover $\XX\to\YY$ is a Kummer cover of degree $k$, with Galois group $L$. Since $L$ is central in $H$, both $b$ and $t$ commute with a generator of $L$. By the last assertion of Lemma~\ref{lem:kummernormalform}, their lifts send a Kummer generator $v$ to a rational function times $v$. It follows that they preserve the exponent residues of the Kummer equation. The map $t$ exchanges $0$ and $\infty$, while $b$ is transitive on the other $M$ ramification points. Therefore the residues at $0$ and $\infty$ are equal, say to $\alpha$, and the residues at the remaining $M$ points are all equal, say to $\beta$.

The cover is totally ramified at these $M+2$ points. Lemma~\ref{lem:kummernormalform} gives
\[
 \gcd(\alpha,k)=\gcd(\beta,k)=1.
\]
The sum of all exponent residues is zero modulo $k$, and hence
\[
 2\alpha+M\beta\equiv0\pmod k.
\]
Since $k$ is odd and $\beta$ is invertible modulo $k$, the integer $2\beta^{-1}$ is invertible modulo $k$. Replacing the Kummer generator by the corresponding power, as in Lemma~\ref{lem:kummernormalform}, changes $\beta$ into $2$. The congruence becomes
\[
 2\alpha+2M\equiv0\pmod k,
\]
and, since $k$ is odd, $\alpha\equiv-M\pmod k$. A rational function with precisely these residues is
\[
 u^{-M}(u^M+1)^2.
\]
Indeed, it has pole order $M$ at both $0$ and $\infty$, and zero order two at every root of $u^M=1$. Two rational functions on $\PP^1$ with the same orders modulo $k$ at every point differ by a $k$th power times a nonzero constant. Since $K$ is algebraically closed, the constant also has a $k$th root. Thus, after changing the Kummer generator, the equation is
\[
 v^k=u^{-M}(u^M+1)^2=u^M+u^{-M}.
\]
Rename $u,v$ as $x,y$ to obtain \eqref{eq:family}.

The equation is unchanged by $u\mapsto\zeta_Mu$ and $u\mapsto u^{-1}$ when $v$ is fixed. A lift of either map that commutes with $L$ sends $v$ to a constant $k$th root of unity times $v$. Choose $b$ as a generator of the subgroup $C_M\le G$. Since $b^M=1$, its multiplier is both a $k$th and an $M$th root of unity. The equality $\gcd(k,M)=1$ forces this multiplier to be one. Similarly, $t^2=1$ and $k$ is odd, so the multiplier of $t$ is one. Finally, $L$ acts by multiplying $v$ by $\zeta_k$. Renaming $u,v$ as $x,y$ gives the group in \eqref{eq:Haction}.

Let $k,M\ge3$ be odd and coprime. In \eqref{eq:family}, the right-hand side is
\[
 x^{-M}(x^M+1)^2.
\]
It has poles of order $M$ at $0,\infty$ and zeros of order two at each of the $M$ distinct roots of $x^M=1$. Since $\gcd(k,M)=1$ and $k$ is odd, all these orders are coprime to $k$. The Kummer equation therefore has degree $k$ and is totally ramified at precisely these $M+2$ points. Riemann--Hurwitz for the extension $K(x,y)/K(x)$ gives
\[
 2g-2=-2k+(M+2)(k-1)=M(k-1)-2.
\]
The three maps $a,b,t$ in Theorem~\ref{thm:index} preserve the equation. The map $a$ multiplies $y$ by a $k$th root of unity and hence commutes with $b,t$. The map $t$ sends $x$ to $x^{-1}$, so $tbt=b^{-1}$. Their restrictions to $K(x)$ give a dihedral group of order $2M$, while the kernel on $K(x)$ is $\langle a\rangle$ of order $k$. Thus the generated group has order $2kM$ and is $C_k\times D_{2M}$. Because $\gcd(k,M)=1$, the subgroup $\langle a,b\rangle$ is cyclic of order $kM$ and has index two. Finally,
\begin{equation}\label{eq:excess}
 |H|-(4g+4)=2kM-\bigl(2M(k-1)+4\bigr)=2M-4>0.
\end{equation}
This proves Theorem~\ref{thm:index}.
\end{proof}

\begin{remark}\label{rem:char2phenomenon}
The family in Theorem~\ref{thm:index} is large because the characteristic is two. To see this, let $K$ have characteristic different from two, with the characteristic not dividing $kM$, and consider the same equation
\[
 y^k=x^M+x^{-M}=\frac{x^{2M}+1}{x^M}.
\]
The polynomial $x^{2M}+1$ has $2M$ distinct roots. Thus the degree-$k$ Kummer cover of the $x$-line is totally ramified over these $2M$ points and over $0,\infty$. Riemann--Hurwitz gives
\[
 2g-2=-2k+(2M+2)(k-1),
\]
and hence
\[
 g=M(k-1).
\]
The maps $a,b,t$ in \eqref{eq:Haction} still generate a group $C_k\times D_{2M}$ of order
\[
 2kM=2g+2M\le3g<4g+4,
\]
where the inequality follows from $k\ge3$, and therefore $M\le g/2$. In characteristic two, however,
\[
 x^{2M}+1=(x^M+1)^2.
\]
The $2M$ simple zeros merge into $M$ double zeros. The Kummer cover is ramified over $M+2$ points, its genus becomes $M(k-1)/2$, and the same group has order $4g+2M>4g+4$.
\end{remark}

\begin{remark}
The classification concerns the triple $(\XX,H,G)$ with its prescribed cyclic subgroup. It gives the displayed action for every pair $(k,M)$, without a continuous parameter. It does not assert that $H=\Aut(\XX_{k,M})$, nor that curves arising from different pairs of the same genus cannot be isomorphic. Proposition~\ref{prop:autcriterion} below gives a criterion for the first question.
\end{remark}

We conclude this section with some information about the ramification of the group in Theorem~\ref{thm:index}. The equation \eqref{eq:family} gives
\[
 z=x^M,\qquad w=z+z^{-1}=y^k.
\]
Then $K(\XX)^G=K(z)$ and $K(\XX)^H=K(w)$: the extensions have degrees $kM$ and $2kM$, respectively. This also exhibits both quotients as rational curves.

\begin{proposition}\label{prop:Hramification}
For $\XX_{k,M}$ and the group in Theorem~\ref{thm:index}, the $G$-cover ramifies over $z=0,\infty,1$, and the corresponding point stabilizers have orders $kM,kM,k$. The $H$-cover is tamely ramified over $w=\infty$, where the stabilizer has order $kM$, and wildly ramified over $w=0$, where the stabilizer has order $2k$. It is unramified elsewhere. At a point $P$ above $w=0$, the ramification groups are
\[
 H_P^{(0)}=C_{2k},\qquad H_P^{(i)}=C_2\ (1\le i\le k),
 \qquad H_P^{(i)}=\{1\}\ (i>k).
\]
\end{proposition}
\begin{proof}
The $G$-cover is the composition of the degree-$k$ Kummer cover $K(x,y)/K(x)$ with the degree-$M$ cover $z=x^M$. At $x=0,\infty$, both covers are totally ramified. Since $k$ and $M$ are coprime, the stabilizer in $G$ there has order $kM$; these points lie over $z=0,\infty$. At $z=1$, the rational function $z+z^{-1}=(z+1)^2/z$ has a zero of order two. Because $k$ is odd, the degree-$k$ cover is totally ramified there, whereas $x^M=z$ is unramified. Its stabilizer therefore has order $k$.

The involution $t$ exchanges $z=0$ and $z=\infty$, and both map to $w=\infty$. The fiber of the $H$-cover over $w=\infty$ therefore consists of two points. Each is fixed by $G$, so its stabilizer has order $kM$; the cover is tamely ramified over $w=\infty$.

Consider now $w=0$. In characteristic two,
\[
 z+z^{-1}=0\quad\Longleftrightarrow\quad z^2+1=(z+1)^2=0,
\]
so $z=1$. Its fiber consists of the $M$ points $(x,y)$ with $x^M=1$ and $y=0$. At $P=(1,0)$, the elements $a$ and $t$ fix both coordinates. Hence
\[
 H_P=\langle a,t\rangle\cong C_k\times C_2=C_{2k}.
\]
The stabilizers at the other $M-1$ points are conjugate to $H_P$.

There is no further short orbit. Indeed, an element of $H\setminus G$ induces the automorphism $x\mapsto\zeta_M^i/x$ of $K(x)$, for some $i$. A fixed point satisfies $x^2=\zeta_M^i$. Raising this equality to the $M$th power gives $(x^M)^2=1$, and characteristic two yields $x^M=1$. The equation of the curve then gives $y=0$. Thus every fixed point of an element in $H\setminus G$ lies in the fiber over $w=0$.

We compute the ramification groups at $P=(1,0)$ using a local parameter. The equation $y^k=x^{-M}(x^M+1)^2$ gives $\ord_P(x+1)=k$ and $\ord_P(y)=2$: the factor $x^M+1$ has a simple zero at $x=1$, and the degree-$k$ Kummer cover is totally ramified there. As $k$ is odd, the function
\[
 v=\frac{y^{(k+1)/2}}{x+1}
\]
has order $2(k+1)/2-k=1$, so it is a local parameter. Since $t$ sends $x$ to $x^{-1}$ and fixes $y$, we find $t(v)=xv$. Hence $\ord_P(t(v)-v)=\ord_P((x+1)v)=k+1$. By the definition of the groups $H_P^{(i)}$, the involution $t$ belongs to them for $1\le i\le k$ and not for $i>k$. The odd-order subgroup $\langle a\rangle$ occurs only in $H_P^{(0)}$. This proves the stated filtration.
\end{proof}

The expression $x^n+x^{-n}$ is a polynomial in $u=x+x^{-1}$. To see this directly, define
\[
 P_0(U)=0,\qquad P_1(U)=U,\qquad
 P_n(U)=U P_{n-1}(U)+P_{n-2}(U).
\]
The recurrence directly gives
\begin{equation}\label{eq:polynomial}
 P_n(x+x^{-1})=x^n+x^{-n}.
\end{equation}

\begin{proposition}\label{prop:involutionquotient}
For $t:(x,y)\mapsto(x^{-1},y)$, the quotient has equation
\begin{equation}\label{eq:quotientmodel}
 \XX_{k,M}/\langle t\rangle:\quad y^k=P_M(u),
 \qquad u=x+x^{-1},
\end{equation}
and
\begin{equation}\label{eq:quotientgenus}
 g(\XX_{k,M}/\langle t\rangle)=\frac{(M-1)(k-1)}4,
 \qquad
 \gamma(\XX_{k,M})=2\gamma(\XX_{k,M}/\langle t\rangle).
\end{equation}
In particular,
\begin{equation}\label{eq:rankbound}
 \gamma(\XX_{k,M})\le g(\XX_{k,M})-\frac{k-1}{2}<g(\XX_{k,M}).
\end{equation}
Thus every curve in Theorem~\ref{thm:index} has even $2$-rank and none is ordinary.
\end{proposition}
\begin{proof}
The functions $u=x+x^{-1}$ and $y$ are fixed by $t$ and satisfy \eqref{eq:quotientmodel} by \eqref{eq:polynomial}. Conversely, $x$ satisfies $x^2+ux+1=0$ over $K(u,y)$. This quadratic equation has degree two in the function field, because $t$ exchanges its roots $x$ and $x^{-1}$ nontrivially. Hence $K(u,y)$ is precisely the fixed field.

A fixed point of $t$ must lie over a solution of $x=x^{-1}$. In characteristic two this equation is $(x+1)^2=0$, so $x=1$ is the only possibility. The Kummer cover is totally ramified there and has a single point above it. Proposition~\ref{prop:Hramification} shows that the higher ramification groups of $\langle t\rangle$ have order two for $1\le i\le k$. Thus its different exponent at the unique fixed point is $(2-1)+k(2-1)=k+1$. Riemann--Hurwitz for this degree-two cover gives
\[
 M(k-1)-2=2(2g(\XX_{k,M}/\langle t\rangle)-2)+(k+1).
\]
Solving this equality gives
\[
 4g(\XX/\langle t\rangle)=(M-1)(k-1),
\]
which is the first assertion of \eqref{eq:quotientgenus}. For the $2$-rank, the one fixed point is the only short orbit of $\langle t\rangle$. Formula \eqref{eq:DS} becomes
\[
 \gamma(\XX)-1=2(\gamma(\XX/\langle t\rangle)-1)+(2-1),
\]
which simplifies to $\gamma(\XX)=2\gamma(\XX/\langle t\rangle)$. Finally, $\gamma(\XX/\langle t\rangle)\le g(\XX/\langle t\rangle)$, whence
\[
 \gamma(\XX)\le\frac{(M-1)(k-1)}2
               =g(\XX)-\frac{k-1}{2}<g(\XX).
\]
This proves \eqref{eq:rankbound}.
\end{proof}

\begin{example}\label{ex:g5}
The curve
\[
 \XX_{3,5}:\quad y^3=x^5+x^{-5}
\]
has genus five and the explicit group $C_3\times D_{10}$ of order $30=6g>4g+4=24$. Its cyclic subgroup of index two has order fifteen. The quotient by the subgroup $\langle t\rangle$, where $t(x,y)=(x^{-1},y)$, has genus two and equation $y^3=u^5+u^3+u$. Proposition~\ref{prop:involutionquotient} gives $\gamma(\XX_{3,5})\le4$, so even this curve at the top of the spectrum is not ordinary.
\end{example}
\section{Spectrum of orders and dihedral groups}
We now determine the possible orders of the groups in Theorem~\ref{thm:index}, prove the bound $6g$, and then study dihedral groups.

We first prove formula \eqref{eq:spectrum}. Equation \eqref{eq:NminusM} gives
\[
 g=M\frac{k-1}{2}.
\]
Since $k$ is odd, $(k-1)/2$ is an integer. Hence $M\mid g$.

Conversely, let $M\mid g$ be odd, with $M\ge3$, and put
\[
 k=\frac{2g}{M}+1.
\]
Then $k$ is odd and $k\ge3$. The pair $(k,M)$ occurs in Theorem~\ref{thm:index} if and only if $\gcd(k,M)=1$. For every such pair, the theorem gives a group of order
\[
 |H|=2kM=4g+2M.
\]
This is exactly the set in \eqref{eq:spectrum}.

We next prove the bound. Since $M\mid g$ and $M\ge3$, we have $M\le g$. Therefore
\[
 |H|=4g+2M\le6g.
\]
Equality holds if and only if $M=g$. In that case
\[
 k=\frac{2g}{M}+1=3.
\]
The conditions of Theorem~\ref{thm:index} now say that $g$ is odd, $g\ge3$, and $3\nmid g$. The first such genus is five. This proves the equality statement in Corollary~\ref{cor:spectrumintro}.

It remains to prove that the curves in Theorem~\ref{thm:index} are not hyperelliptic. Suppose otherwise, and let $\iota$ be the hyperelliptic involution. It is central in $\Aut(\XX)$. The group $G=C_N$ has odd order, so $G\cap\langle\iota\rangle=\{1\}$. Hence
\[
 \langle G,\iota\rangle\cong C_N\times C_2\cong C_{2N}.
\]
Moreover, $N>2g+2$. Thus Theorem~\ref{thm:cyclic} applies to this cyclic group. In case~\textup{(II)}, $2N=4g+2$, so $N=2g+1$. In case~\textup{(III)}, $2N=2g+2$, so $N=g+1$. Both conclusions contradict $N>2g+2$. Therefore $\XX$ is not hyperelliptic.

\begin{corollary}\label{cor:universal}
Every group $H\le\Aut(\XX)$ containing a cyclic subgroup of index two satisfies $|H|\le6g$. Equality holds precisely for the dihedral groups of order twelve on the genus-two curves in Theorem~\ref{thm:dihedral}, and for the groups $C_3\times D_{2g}$ on $\XX_{3,g}$, where $g\ge5$ is odd and $3\nmid g$.
\end{corollary}
\begin{proof}
If $|H|\le4g+4$, then $4g+4\le6g$ for $g\ge2$. Equality in this range requires $g=2$ and $|H|=12$. The cyclic subgroup has order six. By Theorem~\ref{thm:cyclic}, the curve belongs to case~\textup{(III)}: in case~\textup{(II)} the cyclic order in genus two is ten. Proposition~\ref{prop:normalizers} shows that $H$ is the dihedral group of order twelve. If $|H|>4g+4$, Corollary~\ref{cor:spectrumintro} gives the bound and its equality cases.
\end{proof}

\begin{corollary}\label{cor:powers2}
If $g$ is a power of two, then no group with a cyclic subgroup of index two has order greater than $4g+4$.
\end{corollary}
\begin{proof}
By \eqref{eq:spectrum}, an order above $4g+4$ requires an odd divisor $M\mid g$ with $M\ge3$. A power of two has no such divisor.
\end{proof}

For each fixed genus, formula \eqref{eq:spectrum} reduces the classification to a short test on the odd divisors of $g$. The genera $2,3,4,8,10,12,16$ have empty spectrum. Table~\ref{tab:spectrum} gives the nonempty spectra through genus twenty.

\begin{table}[htbp]
\centering
\caption{The complete nonempty spectra for $2\le g\le20$. Each entry is $(k,M;|H|)$.}\label{tab:spectrum}
\begin{tabular}{cl@{\qquad}cl}
\toprule
$g$&$(k,M;|H|)$&$g$&$(k,M;|H|)$\\
\midrule
5&$(3,5;30)$&14&$(5,7;70)$\\
6&$(5,3;30)$&15&$(7,5;70),\ (11,3;66)$\\
7&$(3,7;42)$&17&$(3,17;102)$\\
9&$(7,3;42)$&18&$(5,9;90),\ (13,3;78)$\\
11&$(3,11;66)$&19&$(3,19;114)$\\
13&$(3,13;78)$&20&$(9,5;90)$\\
\bottomrule
\end{tabular}
\end{table}

\begin{proof}[Proof of Theorem~\ref{thm:dihedral}]
Let $D=D_{2N}$ be dihedral. Suppose that $|D|>4g+4$. By Theorem~\ref{thm:index},
\[
 D\cong C_k\times D_{2M},
\]
where $k,M\ge3$ are odd. Let $c$ be a generator of $C_k$ and let $t$ be a reflection. In the direct product, $tc=ct$, so $tct=c$. On the other hand, $c$ belongs to the rotation subgroup of the dihedral group $D$, and the defining relation of a dihedral group gives $tct=c^{-1}$. Hence $c=c^{-1}$, so $c^2=1$. This is impossible because $c$ has odd order $k\ge3$. Therefore $|D|\le4g+4$.

Suppose $|D|=4g+4$. The rotation subgroup has order $N=2g+2$. By Theorem~\ref{thm:cyclic}, $4\nmid N$. Since $N=2(g+1)$, this implies that $g+1$ is odd; hence $g$ is even. Case~\textup{(II)} of that theorem would give $N=4g+2$, which is impossible for $g\ge2$. Thus $\XX$ belongs to case~\textup{(III)}, with $m=g+1$. The rotation subgroup is normal in $D$, so $D$ is contained in its normalizer. Proposition~\ref{prop:normalizers} says that this normalizer is dihedral of order $4m=4g+4$. Hence it is equal to $D$. Conversely, the automorphism in \eqref{eq:reflection} gives equality for every even $g\ge2$ and every $c\ne0$.

We next exclude $|D|=4g+2$. In this case the rotation subgroup is $G=C_N$, where $N=2g+1$ is odd. By Theorem~\ref{thm:cyclic}, the cover $\XX\to\XX/G$ ramifies over three points. A reflection $t$ fixes one of them and exchanges the other two. As in the proof of Theorem~\ref{thm:index}, normalize the exponent residues at the exchanged points to $1$ and $s$. If $\rho$ generates $G$, the dihedral relation is
\[
 t\rho t=\rho^{-1}.
\]
Lemma~\ref{lem:kummernormalform} then gives $s\equiv-1\pmod N$. The residue at the third point is
\[
 -1-s\equiv0\pmod N.
\]
Thus the cover is unramified over that point, a contradiction. Hence a dihedral group cannot have order $4g+2$.

Now let $g$ be odd. We have already proved that equality $|D|=4g+4$ requires $g$ even. We have also excluded $|D|=4g+2$. Since $|D|$ is even, it follows that $|D|\le4g$.

Let $g=m\ge3$ be odd and suppose that $D=D_{4m}$ acts on a curve of genus $m$. Its rotation subgroup is $G=C_{2m}$. Although $2m<2g+1$, we still have $2m>2g-2$, so Lemma~\ref{lem:rational} applies. In particular $\XX/G$ is rational and \eqref{eq:budget} gives
\[
 \sum_Q\delta_Q=2+\frac{2m-2}{2m}=3-\frac1m<3.
\]
By Lemma~\ref{lem:rational}, some stabilizer contains the involution in $G$. Thus $\XX\to\XX/G$ is wildly ramified over at least one point. Each such point contributes at least one, so there are at most two.

Suppose that there is exactly one such point $Q$. Write $2t$ for the order of the stabilizer above $Q$, where $t\mid m$ is odd. Every reflection of $D$ induces an involution on the rational curve $\XX/G$ and must fix $Q$. The fiber over $Q$ contains
\[
 \frac{2m}{2t}=\frac mt
\]
points. This number is odd. The reflection permutes the fiber in orbits of length one or two, so it fixes a point $P$ in the fiber.

The stabilizer $D_P$ contains $G_P=C_{2t}$ and the reflection. Since $[D:G]=2$, it has order $4t$ and is dihedral. If $t>1$, its Sylow $2$-subgroup is not normal: conjugating a reflection by a nontrivial rotation of odd order gives a different reflection. This contradicts the normality of the Sylow $2$-subgroup of a point stabilizer recalled in Section~2. Hence $t=1$.

When $t=1$, the degree-$m$ Kummer cover $\XX/C_2\to\XX/G$ is unramified over $Q$. It must ramify over at least two other points. Each of these points contributes at least $2/3$ to the cover $\XX\to\XX/G$. If there were three, the total would be at least $1+3(2/3)=3$, contrary to the preceding inequality. Hence there are exactly two. Lemma~\ref{lem:kummer} shows that the ramification index at each is $m$ and that $\XX/C_2$ is rational. Using \eqref{eq:threebranchwild} with $g=m$ gives
\[
 2m-2=2m-4+m(j-1),\qquad m(j-1)=2.
\]
Since $m\ge3$, the equation $m(j-1)=2$ has no integer solution. Therefore $\XX\to\XX/G$ is wildly ramified over two points.

Suppose first that one of the two stabilizers has order $2t$ with $t>1$. Its contribution is at least $3/2-1/(2t)\ge4/3$ by Lemma~\ref{lem:local}(1). The other point contributes at least one. If the cover were also tamely ramified over another point, that point would contribute at least $2/3$. The total would be at least three, contrary to \eqref{eq:budget}. Thus there is no further ramification point.

The odd part of the second stabilizer must also be nontrivial. Otherwise the degree-$m$ Kummer cover would ramify over only one point, which is impossible by Lemma~\ref{lem:kummer}. If both stabilizers had order two, the Kummer cover would be unramified over them and would have to ramify over at least two further points. The total contribution to $\XX\to\XX/G$ would then be at least
\[
 2+2\left(\frac23\right)>3,
\]
again impossible. We conclude that the two stabilizers have nontrivial odd part and that there are no other ramification points.

Thus the degree-$m$ Kummer cover ramifies over exactly these two points. Lemma~\ref{lem:kummer} gives ramification index $m$ at both and shows that the quotient by the central involution is rational. Denote the two integers from Lemma~\ref{lem:local}(1) by $j_1,j_2$. The stabilizers for $\XX\to\XX/G$ have order $2m$, so both fibers consist of one point. Riemann--Hurwitz gives
\[
 2m-2=-4m+(2m-1+j_1)+(2m-1+j_2)=j_1+j_2-2.
\]
This equality is $j_1+j_2=2m$. Each $j_i$ is a positive multiple of $m$ by Lemma~\ref{lem:local}(1). Therefore $j_1=j_2=m$.

Choose a coordinate $x$ on the rational quotient $\XX/C_2$ so that the two points over which the Kummer cover ramifies are $x=0,\infty$. By Lemma~\ref{lem:P1cyclic}, the quotient group $G/C_2=C_m$ acts by $x\mapsto\zeta_mx$. A coordinate on $\XX/G$ is $z=x^m$. Lemma~\ref{lem:local}(1) says that the intermediate quadratic Artin--Schreier equation over $K(z)$ has simple poles at $z=0,\infty$ and no others. After removing a constant term, it has the form $y^2+y=Az+B/z$, with $A,B\ne0$. Substituting $z=x^m$ gives
\[
 y^2+y=A x^m+B x^{-m},\qquad A,B\ne0.
\]
Choose a dilation $x\mapsto ax$ such that $Aa^m=B a^{-m}$; this is possible over the algebraically closed field $K$. The two new coefficients are equal and their product is $AB$, so their common value is $c=\sqrt{AB}\ne0$. We obtain
\[
 y^2+y=c(x^m+x^{-m}).
\]
The three automorphisms
\[
 x\mapsto\zeta_mx,\qquad y\mapsto y+1,
 \qquad (x,y)\mapsto(x^{-1},y)
\]
preserve this equation. The first two commute and generate a cyclic group of order $2m$. The third map conjugates a generator of this cyclic group to its inverse. Hence the three maps generate a dihedral group of order $4m$.

It remains to compare this group with the original one. A reflection induces the automorphism $x\mapsto a/x$ of $K(x)$, for some $a\in K^*$. Under this substitution the right-hand side becomes
\[
 c\bigl(a^m x^{-m}+a^{-m}x^m\bigr).
\]
By Lemma~\ref{lem:ASlift}, its difference from $c(x^m+x^{-m})$ must be of the form $h^2+h$. If $a^m\ne1$, this difference has a pole of odd order $m$ at $0$ and at infinity, which is impossible by the last assertion of that lemma. Hence $a^m=1$. Composing the reflection with a suitable rotation changes $a$ to one. The resulting lift of $x\mapsto x^{-1}$ differs from $(x,y)\mapsto(x^{-1},y)$ by at most the central involution $y\mapsto y+1$, which already belongs to the rotation group. Hence the original group is the displayed group.

Finally, the Artin--Schreier equation has two poles, each of order $m$. Formula \eqref{eq:ASgenus} gives
\[
 2g-2=-4+2(m+1)=2m-2,
 \qquad \gamma(\XX)=2-1=1.
\]

It remains to determine the full automorphism group of each curve attaining equality. In both cases put
\[
 \sigma:(x,y)\longmapsto(x,y+1),
 \qquad A=\Aut(\XX).
\]
The fixed field of $\sigma$ is $K(x)$. Hence $\sigma$ is the hyperelliptic involution and, since $g\ge2$, it is central in $A$. It follows that every element of $A$ leaves $K(x)$ invariant. The kernel of the induced homomorphism
\[
 A\longrightarrow\operatorname{Aut}(K(x))\cong\operatorname{PGL}(2,K)
\]
is $\langle\sigma\rangle$. We may therefore regard
\[
 \overline A=A/\langle\sigma\rangle
\]
as a finite subgroup of $\operatorname{PGL}(2,K)$. It preserves the set $B$ of points of the $x$-line over which $K(\XX)/K(x)$ ramifies. If the equation is $y^2+y=f(x)$, Lemma~\ref{lem:ASlift} says that a projectivity $\alpha$ belongs to $\overline A$ if and only if
\[
 f(\alpha(x))+f(x)=h(x)^2+h(x)
\]
for some $h(x)\in K(x)$. The group $\overline A$ contains
\[
 \overline D=\langle x\mapsto\zeta_mx,\ x\mapsto x^{-1}\rangle
 \cong D_{2m},
\]
where $m=g$ in odd genus and $m=g+1$ in even genus.

Assume first that $g$ is odd. Then $f(x)=c(x^m+x^{-m})$ and $B=\{0,\infty\}$. Every element of $\overline A$ has the form $x\mapsto ax$ or $x\mapsto a/x$. For $\alpha(x)=ax$ we have
\[
 f(\alpha(x))+f(x)
 =c(a^m+1)x^m+c(a^{-m}+1)x^{-m}.
\]
If $a^m\ne1$, this function has a pole of odd order $m$ at both $0$ and $\infty$, so it cannot be of the form $h^2+h$. Thus $a^m=1$. The same calculation applies to $\alpha(x)=a/x$. Consequently $\overline A=\overline D$, and therefore $A$ is the displayed dihedral group of order $4m=4g$.

Assume now that $g$ is even. Then $m=g+1$ is odd,
\[
 f(x)=\frac{c}{x^m+1},
 \qquad B=\{x\in K:x^m=1\}.
\]
All points of $B$ are simple poles of $f$.

Let $\tau\in\overline A$ have odd order $n>1$. We claim that $\tau$ fixes no point of $B$. Suppose that it fixes $P\in B$. By Lemma~\ref{lem:P1cyclic}, we can choose a coordinate $u$ on the rational curve such that $P$ is the infinite point and
\[
 \tau(u)=\omega u,
\]
where $\omega$ has order $n$. Since $f$ has a simple pole at $P$, we can write
\[
 f=\lambda u+g(u),
\]
where $\lambda\ne0$ and $g(u)$ is regular at $P$. It follows that
\[
 f(\tau(u))+f(u)=\lambda(\omega+1)u+g(\omega u)+g(u).
\]
The last two terms are regular at $P$, while
\[
 \lambda(\omega+1)u\ne0.
\]
Therefore $f(\tau(u))+f(u)$ has a simple pole at $P$. On the other hand, every pole of $h^2+h$ has even order. This contradicts Lemma~\ref{lem:ASlift} and proves the claim.

Next, let $E\le\overline A$ be an elementary abelian $2$-subgroup. By Lemma~\ref{lem:Dickson}, $E$ has one common fixed point and acts freely outside that point. We now apply the three possibilities in the same lemma.

Suppose first that $\overline A$ is dihedral of order $2n$, with $n$ odd. Its rotation subgroup contains the cyclic group of order $m$ in $\overline D$, so $m\mid n$. On the other hand, the claim above shows that the rotation subgroup acts freely on $B$. Thus every orbit on $B$ has length $n$, and $n\mid |B|=m$. Hence $n=m$ and $\overline A=\overline D$.

Suppose next that $\overline A=E\rtimes C_n$. Since $E$ is normal, its unique fixed point $Q$ is fixed by all of $\overline A$. In particular, it is fixed by the subgroup $C_m$ generated by $x\mapsto\zeta_mx$. Hence $Q$ is $0$ or $\infty$. The claim shows that $Q\notin B$. Therefore $E$ acts freely on $B$, so $|E|$ divides $|B|=m$. This is impossible, because $|E|$ is a power of two greater than one and $m$ is odd.

Finally, suppose that $\overline A\cong\operatorname{PSL}(2,q)$, where $q=2^e$. The case $q=2$ is dihedral and has already been considered, so assume $q\ge4$. By Lemma~\ref{lem:Dickson}, the only orbit of odd length in the standard action on $\PP^1(K)$ is $\PP^1(\FF_q)$, of length $q+1$. Since the invariant set $B$ has odd cardinality, it must contain this orbit. But every point of $\PP^1(\FF_q)$ is fixed by an element of odd order $q-1\ge3$: for example, the stabilizer of the infinite point contains the maps $x\mapsto ax$, with $a\in\FF_q^*$. This contradicts the claim.

The three cases show that $\overline A=\overline D$. Hence $|A|=2|\overline A|=4m=4g+4$. Since $A$ contains the displayed dihedral group of this order, the two groups are equal. This completes the proof.
\end{proof}

\section{Hermitian curves and maximality}\label{sec:hermitian}
Throughout this section $q=2^h$ and $\mathcal H_q$ denotes the Hermitian curve over $\FF_{q^2}$, which we take in the form
\begin{equation}\label{eq:hermitian}
 \mathcal H_q:\quad y^{q+1}=x^q+x .
\end{equation}
It is $\FF_{q^2}$-maximal of genus $q(q-1)/2$, and every curve $\FF_{q^2}$-covered by $\mathcal H_q$ is $\FF_{q^2}$-maximal; see \cite[Chapter 10]{HKT} and \cite{GSX}.

Our first observation is that the Hermitian curve itself belongs to the family $\XX_{k,M}$.

\begin{theorem}\label{thm:hermitianmember}
Let $q=2^h\ge4$. The map $(x,y)\mapsto(x,\,y^2/x)$ is a $K$-isomorphism of $\mathcal H_q$ onto
\[
 \XX_{q+1,q-1}:\quad Y^{q+1}=x^{q-1}+x^{-(q-1)} .
\]
In particular, $\mathcal H_q$ admits the group $C_{q+1}\times D_{2(q-1)}$ of Theorem~\ref{thm:index}, of order $2(q^2-1)=4g+2(q-1)$, whose cyclic subgroup of index two has order $q^2-1=2g+q-1$.
\end{theorem}
\begin{proof}
On $\mathcal H_q$ put $Y=y^2/x$. Squaring \eqref{eq:hermitian} gives $y^{2(q+1)}=(x^q+x)^2$, whence
\[
 Y^{q+1}=\frac{(x^q+x)^2}{x^{q+1}}=\frac{x^{2q}+x^2}{x^{q+1}}=x^{q-1}+x^{-(q-1)} .
\]
Thus $K(x,Y)\subseteq K(x,y)=K(\mathcal H_q)$ and $K(x,Y)$ is the function field of $\XX_{q+1,q-1}$. The field $K(x,y)$ has degree $q+1$ over $K(x)$ by \eqref{eq:hermitian}. The subfield $K(x,Y)$ has the same degree: the function $x^{q-1}+x^{-(q-1)}$ has a pole of order $q-1$ at $x=0$, and
\[
 \gcd(q+1,q-1)=1.
\]
Hence it is not a proper $d$th power for any divisor $d>1$ of $q+1$, so the Kummer equation has degree $q+1$. Therefore $K(x,Y)=K(x,y)$.

The genus of $\XX_{q+1,q-1}$ is $(q-1)q/2$ by \eqref{eq:family}, in accordance with the genus of $\mathcal H_q$. The last assertion is Theorem~\ref{thm:index} with $(k,M)=(q+1,q-1)$.
\end{proof}

The isomorphism is defined over $\FF_2$. It sends the $q+1$ points of $\mathcal H_q$ with $x\in\PP^1(\FF_q)$ to the two points of $\XX_{q+1,q-1}$ above $x=0,\infty$ and to the $q-1$ points above the roots of $x^{q-1}=1$ described in Proposition~\ref{prop:Hramification}. These points are therefore $\FF_{q^2}$-rational.

\begin{remark}
Theorem~\ref{thm:hermitianmember} shows that $\Aut(\XX_{k,M})$ can be larger than the group $H$ of Theorem~\ref{thm:index}: for $(k,M)=(q+1,q-1)$ it is $\mathrm{PGU}(3,q)$. It also shows that the Hermitian curve is an $N$-curve in the sense of \cite{DGTcyclic} with $N=q^2-1\ge2g+1$: this cyclic subgroup of $\mathrm{PGU}(3,q)$ occurs in case~\textup{(I)} of Theorem~\ref{thm:cyclic}. The cover ramifies over three points, with ramification indices $q^2-1,\,q^2-1,\,q+1$.
\end{remark}

We now give two families of $\FF_{q^2}$-maximal curves among the curves $\XX_{k,M}$. The first is obtained by taking quotients of $\XX_{q+1,q-1}$ by subgroups of its cyclic group $G$. The second is obtained from quotients of the Fermat curve of degree $q+1$, another model of $\mathcal H_q$.

For a positive odd integer $n$, denote by $\mu_n$ the group of $n$th roots of unity in $K$. The automorphisms
\[
 (U,V)\longmapsto(\alpha U,\beta V),
 \qquad \alpha,\beta\in\mu_{q+1},
\]
form a group isomorphic to $\mu_{q+1}\times\mu_{q+1}$ on the Fermat curve $U^{q+1}+V^{q+1}=1$. We call it the group of diagonal automorphisms of the Fermat curve.

\begin{proposition}\label{prop:maximalcriteria}
Let $k,M\ge3$ be odd and coprime, and let $q=2^h$. Each of the following conditions implies that $\XX_{k,M}$ is $\FF_{q^2}$-covered by $\mathcal H_q$; in particular $\XX_{k,M}$ is then $\FF_{q^2}$-maximal, with
\begin{equation}\label{eq:maximalcount}
 \#\XX_{k,M}(\FF_{q^2})=q^2+1+2gq,\qquad g=\frac{M(k-1)}2 .
\end{equation}
\begin{enumerate}
\item[\textup{(A)}] $k\mid q+1$ and $M\mid q-1$. In this case $\XX_{k,M}$ is the quotient of $\XX_{q+1,q-1}\cong\mathcal H_q$ by the subgroup $C_{(q+1)/k}\times C_{(q-1)/M}$ of $G$. When $M=q-1$, the curve $\XX_{k,q-1}$ is isomorphic to the curve $y^k=x^q+x$ of \cite{GSX}.
\item[\textup{(B)}] $kM\mid q+1$. In this case $\XX_{k,M}$ is a quotient of the Fermat curve $U^{q+1}+V^{q+1}=1$ by a subgroup of order $(q+1)^2/(kM)$ of its group of diagonal automorphisms $\mu_{q+1}\times\mu_{q+1}$.
\end{enumerate}
Moreover, in both cases the group $C_k\times D_{2M}$ of Theorem~\ref{thm:index} is defined over $\FF_{q^2}$.
\end{proposition}
\begin{proof}
(A) Put $d_1=(q+1)/k$ and $d_2=(q-1)/M$. On $\XX_{q+1,q-1}$, with coordinates $(x,Y)$ as in Theorem~\ref{thm:hermitianmember}, consider $x'=x^{d_2}$ and $Y'=Y^{d_1}$. Then
\[
 Y'^{\,k}=Y^{q+1}=x^{q-1}+x^{-(q-1)}=x'^{\,M}+x'^{\,-M},
\]
so $(x,Y)\mapsto(x',Y')$ is a nonconstant $\FF_{2}$-morphism $\XX_{q+1,q-1}\to\XX_{k,M}$.

Let $S$ be the subgroup generated by $Y\mapsto\zeta_{d_1}Y$ and $x\mapsto\zeta_{d_2}x$. Then $S\cong C_{d_1}\times C_{d_2}$, and both $x'$ and $Y'$ are fixed by $S$. Moreover, $x$ satisfies $T^{d_2}-x'=0$ and, once $x$ is fixed, $Y$ satisfies $T^{d_1}-Y'=0$. Hence
\[
 [K(x,Y):K(x',Y')]\le d_1d_2=|S|.
\]
Since $K(x',Y')\subseteq K(x,Y)^S$, we also have
\[
 [K(x,Y):K(x',Y')]\ge [K(x,Y):K(x,Y)^S]=|S|.
\]
Together with the opposite inequality proved above, this gives equality and
\[
 K(x',Y')=K(x,Y)^S.
\]
Composing with the isomorphism of Theorem~\ref{thm:hermitianmember} gives the covering by $\mathcal H_q$. For $M=q-1$ the same computation as in that theorem, applied to the curve $y^k=x^q+x$ and to $Y=y^2/x^{(q+1)/k}$, gives $Y^k=(x^q+x)^2/x^{q+1}=x^{q-1}+x^{-(q-1)}$, and a degree comparison as before shows that this is an isomorphism.

(B) Put $d=(q+1)/(kM)$ and consider the Fermat curve
\[
 \mathcal F:\quad U^{q+1}+V^{q+1}=1.
\]
On $\mathcal F$, set
\[
 x=U^{kd},\qquad y=\frac{V^{2Md}}{U^{Md}}.
\]
Since $kMd=q+1$ and the characteristic is two,
\[
 y^k=\frac{V^{2(q+1)}}{U^{q+1}}=\frac{(1+U^{q+1})^2}{U^{q+1}}=U^{q+1}+U^{-(q+1)}=x^M+x^{-M}.
\]
The function $x^M+x^{-M}$ has a pole of order $M$ at $x=0$. Since $\gcd(k,M)=1$, it is not a proper $r$th power in $K(x)$ for any divisor $r>1$ of $k$. Hence the Kummer equation has degree $k$, and $K(x,y)$ is the function field of $\XX_{k,M}$.

Let $T$ be the subgroup of the diagonal group $\mu_{q+1}\times\mu_{q+1}$ consisting of the pairs $(\alpha,\beta)$ satisfying
\[
 \alpha^{kd}=1,\qquad \beta^{2Md}=\alpha^{Md}.
\]
These are exactly the diagonal automorphisms fixing both $x$ and $y$. There are $kd$ choices for $\alpha$. We check that, for each of them, the second equation has a solution. Since $\alpha^{kd}=1$, we have
\[
 (\alpha^{Md})^k=\alpha^{kMd}=\alpha^{q+1}=1,
\]
so $\alpha^{Md}\in\mu_k$. On the other hand, the image of
\[
 \mu_{q+1}\longrightarrow\mu_{q+1},
 \qquad \beta\longmapsto\beta^{2Md},
\]
has order
\[
 \frac{q+1}{\gcd(2Md,q+1)}=k.
\]
It is therefore the unique subgroup $\mu_k$ of $\mu_{q+1}$. Thus the equation $\beta^{2Md}=\alpha^{Md}$ is solvable. Its set of solutions is a coset of the kernel, whose order is
\[
 \gcd(2Md,q+1)=Md.
\]
Hence there are $Md$ choices for $\beta$, and therefore
\[
 |T|=kd\cdot Md=\frac{(q+1)^2}{kM}.
\]

We now verify that $K(x,y)$ is the full fixed field of $T$. First, $U$ satisfies $U^{kd}=x$, so there are at most $kd$ possibilities for $U$ over $K(x,y)$. Once $U$ is fixed, put $W=V^{Md}$. The definitions give
\[
 W^2=yU^{Md},\qquad W^k=V^{q+1}=1+U^{q+1}.
\]
Since $k$ is odd, $\gcd(2,k)=1$. Choose integers $r,s$ such that $2r+ks=1$. Then
\[
 W=(W^2)^r(W^k)^s,
\]
so $W\in K(x,y,U)$. Finally, the equation $V^{Md}=W$ has at most $Md$ solutions. Consequently
\[
 [K(U,V):K(x,y)]\le kd\cdot Md=|T|.
\]
On the other hand, $K(x,y)\subseteq K(U,V)^T$, and Galois theory gives
\[
 [K(U,V):K(U,V)^T]=|T|.
\]
The two inequalities force $K(x,y)=K(U,V)^T$. Hence the map is the quotient by $T$. The Fermat curve of degree $q+1$ is $\FF_{q^2}$-isomorphic to $\mathcal H_q$ \cite[Chapter~10]{HKT}.

In both cases the covering is defined over $\FF_{q^2}$ and $\mathcal H_q$ is $\FF_{q^2}$-maximal. Hence $\XX_{k,M}$ is $\FF_{q^2}$-maximal \cite[Chapter 10]{HKT}, \cite{GSX}. Finally, in both cases $k,M\mid q^2-1$, so $\zeta_k,\zeta_M\in\FF_{q^2}$ and the maps $a,b,t$ of Theorem~\ref{thm:index} are defined over $\FF_{q^2}$.
\end{proof}

\begin{remark}\label{rem:criteria}
The two conditions are independent. The pair $(5,7)$ satisfies (A) with $q=64$, but never (B), since $7\nmid2^h+1$ for every $h$. The pair $(3,11)$ satisfies (B) with $q=32$, but never (A), since $3\mid2^h+1$ forces $h$ odd while $11\mid2^h-1$ forces $10\mid h$.

Neither condition can be weakened to $k\mid q+1$ and $M\mid q^2-1$. Consider $(k,M)=(3,341)$ and $q=32$. Then $3\mid33$ and $341=11\cdot31\mid q^2-1=1023$. For $x\in\FF_{1024}^*$, put $z=x^{341}$. Then $z^3=1$. If $z=1$, the right-hand side $z+z^{-1}$ is zero; there are $341$ such values of $x$, and for each of them $y=0$. If $z$ is one of the other two cubic roots of unity, then $z+z^{-1}=1$; for each of the $2\cdot341$ corresponding values of $x$, the equation $y^3=1$ has three solutions. Adding the points over $x=0$ and $x=\infty$, we obtain
\[
 \#\XX_{3,341}(\FF_{1024})
 =341+2\cdot341\cdot3+2=2389,
\]
whereas the maximal value given by \eqref{eq:maximalcount} would be $22849$. We do not know whether every maximal member of the family satisfies (A) or (B).

\end{remark}

\begin{example}\label{ex:hermitian}
Some cases of Proposition~\ref{prop:maximalcriteria} are listed below. The numbers of rational points are obtained from \eqref{eq:maximalcount}.
\begin{center}
\begin{tabular}{@{}cccccl@{}}
\toprule
$(k,M)$ & $g$ & $|H|$ & $q$ & $\#\XX_{k,M}(\FF_{q^2})$ & condition \\
\midrule
$(5,3)$ & $6$ & $30$ & $4$ & $65$ & $\mathcal H_4$ itself \\
$(3,7)$ & $7$ & $42=6g$ & $8$ & $177$ & (A); $y^3=x^8+x$ \\
$(3,11)$ & $11$ & $66=6g$ & $32$ & $1729$ & (B) \\
$(5,7)$ & $14$ & $70$ & $64$ & $5889$ & (A) \\
$(11,3)$ & $15$ & $66$ & $32$ & $1985$ & (B) \\
$(13,3)$, $(5,9)$ & $18$ & $78$, $90$ & $64$ & $6401$ & (A) \\
$(3,19)$ & $19$ & $114=6g$ & $512$ & $281601$ & (B), $57\mid2^9+1$ \\
$(5,13)$ & $26$ & $130$ & $64$ & $7425$ & (B) \\
$(9,7)$ & $28$ & $126$ & $8$ & $513$ & $\mathcal H_8$ itself \\
\bottomrule
\end{tabular}
\end{center}
At the top of the spectrum there are infinitely many maximal curves: for $h\ge5$ odd with $3\nmid h$, put $M=(2^h+1)/3$; then $3\nmid M$, $M$ is odd, and $\XX_{3,M}$ has genus $M$, a group of order $6g$, and is $\FF_{2^{2h}}$-maximal by (B). Likewise, for every $h\ge3$ the curve $\XX_{3,2^h-1}\cong\{y^3=x^{2^h}+x\}$ is $\FF_{2^{2h}}$-maximal by (A) whenever $3\mid 2^h+1$, that is, whenever $h$ is odd.
\end{example}

We now record what can be said about the full automorphism group.

\begin{proposition}\label{prop:autcriterion}
For odd coprime $k,M\ge3$, the following are equivalent:
\begin{enumerate}
\item $\Aut(\XX_{k,M})=H$, where $H\cong C_k\times D_{2M}$ is the group in Theorem~\ref{thm:index};
\item the cyclic group $G=\langle a,b\rangle$ is normal in $\Aut(\XX_{k,M})$.
\end{enumerate}
Both conditions fail whenever $M=2^h-1$ and $k\mid2^h+1$; in particular they fail for the Hermitian curve $\XX_{q+1,q-1}$.
\end{proposition}
\begin{proof}
(1) implies (2) since $G$ has index two in $H$. Assume (2). Every element of $\Aut(\XX)$ induces an automorphism of $\XX/G\cong\PP^1$, with coordinate $z=x^M$. The kernel of the resulting homomorphism is the Galois group $G$ of $K(\XX)/K(z)$. The induced automorphisms preserve the ramification data of Proposition~\ref{prop:Hramification}: the points $z=0,\infty$ have stabilizers of order $kM$, while $z=1$ has stabilizer of order $k<kM$. Hence they fix $z=1$ and preserve $\{0,\infty\}$. A projectivity fixing $1$ and preserving $\{0,\infty\}$ is $z\mapsto z$ or $z\mapsto z^{-1}$. Therefore $|\Aut(\XX)|\le2|G|=|H|$, and (1) follows.

If $M=2^h-1$ and $k\mid2^h+1$, Proposition~\ref{prop:maximalcriteria}(A) identifies $\XX_{k,M}$ with $y^k=x^q+x$, where $q=2^h$. For every $c\in\FF_q$ we have $c^q+c=0$. Hence
\[
 (x,y)\longmapsto(x+c,y)
\]
preserves the equation. These automorphisms form an elementary abelian group of order $q\ge4$. Since $H$ has a Sylow $2$-subgroup of order two, $\Aut(\XX_{k,M})\ne H$.
\end{proof}

\begin{remark}
Several natural questions are left open by the preceding arguments. They do not determine whether $G$ is normal in $\Aut(\XX_{k,M})$ outside the subfamily in the last part of Proposition~\ref{prop:autcriterion}; the curve $\XX_{3,5}$ is the first case to examine. They also do not determine whether every maximal member of the family satisfies condition~\textup{(A)} or~\textup{(B)} of Proposition~\ref{prop:maximalcriteria}, or whether two different admissible pairs $(k,M)$ of the same genus can give isomorphic curves.
\end{remark}

\section*{Statements and Declarations}

\noindent\textbf{Funding.}
No funding was received for conducting this study.

\medskip

\noindent\textbf{Competing interests.}
The author has no relevant financial or non-financial interests to disclose.

\medskip

\noindent\textbf{Data availability.}
Data sharing is not applicable to this article as no datasets were
generated or analysed during the current study.

\section*{Acknowledgements}
The authors thank the Italian National Group for Algebraic and Geometric Structures and their Applications (GNSAGA—INdAM)
which supported the research.

\end{document}